\documentclass[12pt]{amsart}
\usepackage{txfonts}      
\usepackage{amssymb}
\usepackage{eucal}
\usepackage{amsmath}
\usepackage{amscd}
\usepackage{xcolor}
\usepackage{multicol}
\usepackage[all]{xy}           
\usepackage{graphicx}
\usepackage{color}
\usepackage{colordvi}
\usepackage{xspace}
\usepackage{tikz}
\usepackage{makecell}
\usepackage{appendix}
\usepackage{amsthm}
\usepackage[misc]{ifsym}
\usepackage{mathrsfs} 
\usepackage{xypic}
\usepackage{extarrows}

\usepackage{ifpdf}
\ifpdf
\usepackage[colorlinks,final,backref=page,hyperindex]{hyperref}
\else
\usepackage[colorlinks,final,backref=page,hyperindex,hypertex]{hyperref}
\fi

\usepackage[active]{srcltx} 

\begin{document}
    \newtheorem{theorem}{Theorem}[section]
    \newtheorem{lemma}[theorem]{Lemma}
    \newtheorem{corollary}[theorem]{Corollary}
    \newtheorem{proposition}[theorem]{Proposition}
    \theoremstyle{definition}
    \newtheorem{definition}[theorem]{Definition}
    \newtheorem{example}[theorem]{Example}
    \newtheorem{remark}[theorem]{Remark}
    \newtheorem{pdef}[theorem]{Proposition-Definition}
    \newtheorem{condition}[theorem]{Condition}
    \renewcommand{\labelenumi}{{\rm(\alph{enumi})}}
    \renewcommand{\theenumi}{\alph{enumi}}
    \baselineskip=14pt

    \newcommand {\emptycomment}[1]{} 

    \newcommand{\nc}{\newcommand}
    \newcommand{\delete}[1]{}

    \nc{\todo}[1]{\tred{To do:} #1}

    \nc{\tred}[1]{\textcolor{red}{#1}}
    \nc{\tblue}[1]{\textcolor{blue}{#1}}
    \nc{\tgreen}[1]{\textcolor{green}{#1}}
    \nc{\tpurple}[1]{\textcolor{purple}{#1}}
    \nc{\tgray}[1]{\textcolor{gray}{#1}}
    \nc{\torg}[1]{\textcolor{orange}{#1}}
    \nc{\tmag}[1]{\textcolor{magenta}}
    \nc{\btred}[1]{\textcolor{red}{\bf #1}}
    \nc{\btblue}[1]{\textcolor{blue}{\bf #1}}
    \nc{\btgreen}[1]{\textcolor{green}{\bf #1}}
    \nc{\btpurple}[1]{\textcolor{purple}{\bf #1}}

        \nc{\mlabel}[1]{\label{#1}}  
        \nc{\mcite}[1]{\cite{#1}}  
        \nc{\mref}[1]{\ref{#1}}  
        \nc{\meqref}[1]{\eqref{#1}}  
        \nc{\mbibitem}[1]{\bibitem{#1}} 

    \delete{
        \nc{\mlabel}[1]{\label{#1}  
            { {\small\tgreen{\tt{{\ }(#1)}}}}}
        \nc{\mcite}[1]{\cite{#1}{\small{\tt{{\ }(#1)}}}}  
        \nc{\mref}[1]{\ref{#1}{\small{\tred{\tt{{\ }(#1)}}}}}  
        \nc{\meqref}[1]{\eqref{#1}{{\tt{{\ }(#1)}}}}  
        \nc{\mbibitem}[1]{\bibitem[\bf #1]{#1}} 
    }

    \nc{\cm}[1]{\textcolor{red}{Chengming:#1}}
    \nc{\yy}[1]{\textcolor{blue}{Yanyong: #1}}
    \nc{\zy}[1]{\textcolor{yellow}{Zhongyin: #1}}
    \nc{\li}[1]{\textcolor{purple}{#1}}
    \nc{\lir}[1]{\textcolor{purple}{Li:#1}}


    \nc{\tforall}{\ \ \text{for all }}
    \nc{\hatot}{\,\widehat{\otimes} \,}
    \nc{\complete}{completed\xspace}
    \nc{\wdhat}[1]{\widehat{#1}}

    \nc{\ts}{\mathfrak{p}}
    \nc{\mts}{c_{(i)}\ot d_{(j)}}

    \nc{\NA}{{\bf NA}}
    \nc{\LA}{{\bf Lie}}
    \nc{\CLA}{{\bf CLA}}

    \nc{\cybe}{CYBE\xspace}
    \nc{\nybe}{NYBE\xspace}
    \nc{\ccybe}{CCYBE\xspace}
    \nc{\rk}{\mathrm{r}}
    \nc{\ndend}{pre-Novikov\xspace}
    \newcommand{\B}{\mathfrak{B}}
    \newcommand{\T}{\mathfrak{T}}
    \newcommand{\A}{\mathcal{A}}
    \newcommand{\C}{\mathfrak{C}}
    \newcommand{\g}{\mathfrak g}
    \newcommand{\h}{\mathfrak h}
    \newcommand{\cP}{\mathfrak{P}}
    \newcommand{\Wh}{\widehat}
    \newcommand{\id}{\mathrm{id}}

    \newcommand{\pf}{\noindent{$Proof$.}\ }
    \newcommand{\gl}{\mathfrak {gl}}
    \newcommand{\ad}{\mathfrak{ad}}
    \newcommand{\D}{\mathcal{D}}
    \newcommand{\frka}{\mathfrak a}
    \newcommand{\frkb}{\mathfrak b}
    \newcommand{\frkc}{\mathfrak c}
    \newcommand{\frkd}{\mathfrak d}
    \newcommand{\rb}{ \rho_{\mathfrak{B},\mathfrak{B}'}}
    \newcommand{\sib}{ \sigma_{\mathfrak{B},\mathfrak{B}'}}
    \newcommand {\comment}[1]{{\marginpar{*}\scriptsize\textbf{Comments:} #1}}


    \nc{\disp}[1]{\displaystyle{#1}}
    \nc{\bin}[2]{ (_{\stackrel{\scs{#1}}{\scs{#2}}})}  
    \nc{\binc}[2]{ \left (\!\! \begin{array}{c} \scs{#1}\\
            \scs{#2} \end{array}\!\! \right )}  
    \nc{\bincc}[2]{  \left ( {\scs{#1} \atop
            \vspace{-.5cm}\scs{#2}} \right )}  
    \nc{\ot}{\otimes}
    \nc{\sot}{{\scriptstyle{\ot}}}
    \nc{\otm}{\overline{\ot}}
    \nc{\ola}[1]{\stackrel{#1}{\la}}

    \nc{\scs}[1]{\scriptstyle{#1}} \nc{\mrm}[1]{{\rm #1}}

    \nc{\dirlim}{\displaystyle{\lim_{\longrightarrow}}\,}
    \nc{\invlim}{\displaystyle{\lim_{\longleftarrow}}\,}

    \nc{\bfk}{{\bf k}} \nc{\bfone}{{\bf 1}}
    \nc{\rpr}{\circ}
    \nc{\dpr}{{\tiny\diamond}}
    \nc{\rprpm}{{\rpr}}

    \nc{\mmbox}[1]{\mbox{\ #1\ }} \nc{\ann}{\mrm{ann}}
    \nc{\Aut}{\mrm{Aut}} \nc{\can}{\mrm{can}}
    \nc{\twoalg}{{two-sided algebra}\xspace}
    \nc{\colim}{\mrm{colim}}
    \nc{\Cont}{\mrm{Cont}} \nc{\rchar}{\mrm{char}}
    \nc{\cok}{\mrm{coker}} \nc{\dtf}{{R-{\rm tf}}} \nc{\dtor}{{R-{\rm
                tor}}}
    \renewcommand{\det}{\mrm{det}}
    \nc{\depth}{{\mrm d}}
    \nc{\End}{\mrm{End}} \nc{\Ext}{\mrm{Ext}}
    \nc{\Fil}{\mrm{Fil}} \nc{\Frob}{\mrm{Frob}} \nc{\Gal}{\mrm{Gal}}
    \nc{\GL}{\mrm{GL}} \nc{\Hom}{\mrm{Hom}} \nc{\hsr}{\mrm{H}}

    \nc{\incl}{\mrm{incl}} \nc{\length}{\mrm{length}}
    \nc{\LR}{\mrm{LR}} \nc{\mchar}{\rm char} \nc{\NC}{\mrm{NC}}
    \nc{\mpart}{\mrm{part}} \nc{\pl}{\mrm{PL}}
    \nc{\ql}{{\QQ_\ell}} \nc{\qp}{{\QQ_p}}
    \nc{\rank}{\mrm{rank}} \nc{\rba}{\rm{Rota-Baxter A }} \nc{\rbas}{\rm{Rota-Baxter As }}
    \nc{\rbpl}{\mrm{Rota-Baxter PL}}
    \nc{\rbw}{\rm{Rota-Baxter W }} \nc{\rbws}{\rm{Rota-Baxter Ws }} \nc{\rcot}{\mrm{cot}}
    \nc{\rest}{\rm{controlled}\xspace}
    \nc{\rdef}{\mrm{def}} \nc{\rdiv}{{\rm div}} \nc{\rtf}{{\rm tf}}
    \nc{\rtor}{{\rm tor}} \nc{\res}{\mrm{res}} \nc{\SL}{\mrm{SL}}
    \nc{\Spec}{\mrm{Spec}} \nc{\tor}{\mrm{tor}} \nc{\Tr}{\mrm{Tr}}
    \nc{\mtr}{\mrm{sk}}

    \nc{\ab}{\mathbf{Ab}} \nc{\Alg}{\mathbf{Alg}}

    \nc{\BA}{{\mathbb A}} \nc{\CC}{{\mathbb C}} \nc{\DD}{{\mathbb D}}
    \nc{\EE}{{\mathbb E}} \nc{\FF}{{\mathbb F}} \nc{\GG}{{\mathbb G}}
    \nc{\HH}{{\mathbb H}} \nc{\LL}{{\mathbb L}} \nc{\NN}{{\mathbb N}}
    \nc{\QQ}{{\mathbb Q}} \nc{\RR}{{\mathbb R}} \nc{\BS}{{\mathbb{S}}} \nc{\TT}{{\mathbb T}}
    \nc{\VV}{{\mathbb V}} \nc{\ZZ}{{\mathbb Z}}


    \nc{\calao}{{\mathcal A}} \nc{\cala}{{\mathcal A}}
    \nc{\calc}{{\mathcal C}} \nc{\cald}{{\mathcal D}}
    \nc{\cale}{{\mathcal E}} \nc{\calf}{{\mathcal F}}
    \nc{\calfr}{{{\mathcal F}^{\,r}}} \nc{\calfo}{{\mathcal F}^0}
    \nc{\calfro}{{\mathcal F}^{\,r,0}} \nc{\oF}{\overline{F}}
    \nc{\calg}{{\mathcal G}} \nc{\calh}{{\mathcal H}}
    \nc{\cali}{{\mathcal I}} \nc{\calj}{{\mathcal J}}
    \nc{\call}{{\mathcal L}} \nc{\calm}{{\mathcal M}}
    \nc{\caln}{{\mathcal N}} \nc{\calo}{{\mathcal O}}
    \nc{\calp}{{\mathcal P}} \nc{\calq}{{\mathcal Q}} \nc{\calr}{{\mathcal R}}
    \nc{\calt}{{\mathcal T}} \nc{\caltr}{{\mathcal T}^{\,r}}
    \nc{\calu}{{\mathcal U}} \nc{\calv}{{\mathcal V}}
    \nc{\calw}{{\mathcal W}} \nc{\calx}{{\mathcal X}}
    \nc{\CA}{\mathcal{A}}

    \nc{\fraka}{{\mathfrak a}} \nc{\frakB}{{\mathfrak B}}
    \nc{\frakb}{{\mathfrak b}} \nc{\frakd}{{\mathfrak d}}
    \nc{\oD}{\overline{D}}
    \nc{\frakF}{{\mathfrak F}} \nc{\frakg}{{\mathfrak g}}
    \nc{\frakm}{{\mathfrak m}} \nc{\frakM}{{\mathfrak M}}
    \nc{\frakMo}{{\mathfrak M}^0} \nc{\frakp}{{\mathfrak p}}
    \nc{\frakS}{{\mathfrak S}} \nc{\frakSo}{{\mathfrak S}^0}
    \nc{\fraks}{{\mathfrak s}} \nc{\os}{\overline{\fraks}}
    \nc{\frakT}{{\mathfrak T}}
    \nc{\oT}{\overline{T}}
    \nc{\frakX}{{\mathfrak X}} \nc{\frakXo}{{\mathfrak X}^0}
    \nc{\frakx}{{\mathbf x}}
    \nc{\frakTx}{\frakT}      
    \nc{\frakTa}{\frakT^a}        
    \nc{\frakTxo}{\frakTx^0}   
    \nc{\caltao}{\calt^{a,0}}   
    \nc{\ox}{\overline{\frakx}} \nc{\fraky}{{\mathfrak y}}
    \nc{\frakz}{{\mathfrak z}} \nc{\oX}{\overline{X}}

    \font\cyr=wncyr10


    \title[Factorizable Lie conformal algebras]{Factorizable Lie conformal bialgebras, quadratic Rota-Baxter Lie conformal algebras and some induced structures}

    \author{Zhongyin Xu}
    \address{Chern Institute of Mathematics \& LPMC, Nankai University,
        Tianjin, 300071, China}
    \email{zy\_xu@mail.nankai.edu.cn}

    \author{Chengming Bai}
    \address{Chern Institute of Mathematics \& LPMC, Nankai University, Tianjin 300071, PR China}
    \email{baicm@nankai.edu.cn}

    \author{Yanyong Hong}
    \address{School of Mathematics, Hangzhou Normal University,
        Hangzhou, 311121, China}
    \email{yyhong@hznu.edu.cn}

    \subjclass[2010]{17A30, 17A60, 17B60, 17B69, 17D25}
    \keywords{Lie conformal bialgebra, Gel'fand-Dorfman bialgebra, factorizable Lie conformal bialgebra, Rota-Baxter Lie conformal algebra, factorizable Gel'fand-Dorfman bialgebra, Rota-Baxter Gel'fand-Dorfman algebra }
    \begin{abstract}
        We introduce the notion of factorizable Lie conformal bialgebras  and establish their correspondence with quadratic Rota-Baxter Lie conformal algebras of nonzero weight.
Consequently, on the one hand, a Lie conformal bialgebra  with a
Rota-Baxter operator of nonzero weight gives a factorizable Lie
conformal bialgebra. And on the other hand, as the conformal
analogue of the fact that a quadratic Rota-Baxter Lie algebra
induces a generalized pseudo-Hessian post-Lie algebra and a
special partial-pre-post-Lie algebra, a quadratic Rota-Baxter Lie
conformal algebra as well as a factorizable Lie conformal
bialgebra induces a generalized pseudo-Hessian post-Lie conformal
algebra and a special partial-pre-post-Lie conformal algebra.
Furthermore, we generalize the correspondence between
Gel'fand-Dorfman bialgebras and  a class of Lie conformal
bialgebras to the factorizable cases. There is not only a
correspondence between factorizable Gel'fand-Dorfman bialgebras
and a class of factorizable Lie conformal bialgebras, but also a
correspondence for their corresponding quadratic Rota-Baxter
counterparts as well as some induced structures. In particular,
there is a construction of factorizable Lie conformal bialgebras
from Gel'fand-Dorfman bialgebras  with a Rota-Baxter operator of
nonzero weight.
    \end{abstract}

    \maketitle
    \tableofcontents

    \section{Introduction}
    The notion of Lie conformal algebras was introduced by V. Kac in \cite{K1,K2} as a formal language describing the algebraic properties of the singular
    parts of operator product expansions of chiral fields in two-dimensional conformal field theory. They are closely related to vertex algebras \cite{K1}, infinite-dimensional Lie algebras that satisfy
    the locality property \cite{K3}, and the Hamiltonian formalism used in the theory of nonlinear evolution equations \cite{BDK}. As the conformal analog of Lie bialgebras,  Lie conformal bialgebras were introduced  in \cite{L}, together with conformal Manin triples, and the conformal classical Yang-Baxter equation (CCYBE).

    The classical Yang-Baxter equation (CYBE), which is closely related to classical integrable systems and quantum groups \cite{CP,STS},  giving the notion of quasi-triangular
     Lie bialgebras.  In addition, factorizable Lie bialgebras introduced in \cite{RS} are a special class of quasi-triangular Lie bialgebras, which are used to establish the
     relation between the solutions of the CYBE and certain factorization problems in integrable systems \cite{STS2}.

    The notion of Rota-Baxter operators on associative algebras was introduced in the probability study of G. Baxter in \cite{B}. Motivated by the connections in combinatorics \cite{R}, 
    quantum field theory \cite{CK2}, noncommutative symmetric functions \cite{YGT}, 
    Poisson geometry \cite{U,S} 
    and operads \cite{BBGN}, 
    the study of Rota-Baxter operators has
    become increasingly active in recent years. As a remarkable coincidence, Rota-Baxter operators on Lie algebras also arose independently as the  operator form
    of the CYBE. Moreover, there is a correspondence
    between factorizable Lie bialgebras and quadratic Rota-Baxter Lie algebras of nonzero weight
    \cite{LS}. In particular, a Rota-Baxter
Lie bialgebra with a nonzero weight \cite{LS} gives a factorizable
Lie bialgebra. Note that there is a notion of Rota-Baxter Lie
bialgebras with a pair of operators satisfying certain admissible
conditions \cite{BGLM}, including the one in \cite{LS} as a
special case. Therefore, it is natural to consider the conformal
analog of
 these results, which is the first motivation of
this paper.

 Furthermore, Novikov algebras are a special class of pre-Lie algebras (or left-symmetric algebras), introduced in connection with Hamiltonian operators in the formal variational calculus \cite{GD1,GD2}
 and Poisson brackets of hydrodynamic type \cite{BN}. A Gel'fand-Dorfman algebra (GDA) is a vector space equipped with a Novikov algebra structure and a Lie algebra structure satisfying a compatibility condition.
 Note that Novikov algebras are GDAs with the trivial Lie algebra structure. GDAs correspond to Hamiltonian pairs, which are fundamental in the study of completely integrable systems \cite{GD1}. 
One of the important roles of GDAs is illustrated by their
correspondence with a class of Lie conformal algebras \cite{X}.
Such a relationship is further lifted to the level of bialgebras
in \cite{HBG} by developing a bialgebra theory  of
Gel'fand-Dorfman algebras, namely  Gel'fand-Dorfman bialgebras
(GDBAs). Therefore,        it is also natural to consider this
correspondence in the factorizable case, which is the second
motivation of this paper.

   For the former, we introduce the notion of factorizable Lie conformal bialgebras and establish their correspondence with quadratic Rota-Baxter Lie conformal algebras of nonzero
   weight. There are two important consequences. One is that, due
   to the aforementioned conformal Manin triple interpretation of
   Lie conformal bialgebras, a Lie conformal bialgebra with a Rota-Baxter operator of nonzero
   weight naturally gives a quadratic Rota-Baxter conformal
   algebra and hence a factorizable Lie conformal bialgebra.
   Another one is involved the induced structures. The actions of
   Rota-Baxter operators on an algebra structure are to split the operations
   of the algebra structure into the sum of two operations or
   three operations, that is, to give the bisuccessors or
   trisuccessors in the sense of operads \cite{BBGN}. For example,
a Rota-Baxter Lie algebra of  weight 1 induces a
   post-Lie algebra. In addition,
a quadratic Rota-Baxter Lie algebra of weight 1 induces a
generalized pseudo-Hessian post-Lie algebra, which finally gives a
special partial-pre-post-Lie algebra \cite{LBG}. As the conformal
analog, a quadratic Rota-Baxter Lie conformal algebra of
   weight 1 also induces  a generalized
pseudo-Hessian post-Lie conformal algebra and a special
partial-pre-post-Lie conformal algebra. Hence a factorizable Lie
conformal bialgebra induces these two structures, too.


For the latter,  we introduce the notion of factorizable GDBAs and
show that they correspond to a class of factorizable Lie conformal
bialgebras. Similarly, quadratic Rota-Baxter GDAs of nonzero
weight correspond to a class of quadratic Rota-Baxter Lie
conformal algebras of the same weight. Moreover, there is also a
correspondence between factorizable GDBAs and quadratic
Rota-Baxter GDAs of nonzero weight so that the following diagram
is commutative:
 \begin{equation*}
        \hspace*{-1.5cm}%
        \scalebox{0.85}{
            {\small \xymatrix@C=1.5cm{
                    &\txt{ factorizable GDBAs}\ar[d]\ar@<-1ex>[r]
                    & \txt{  quadratic    Rota-Baxter GDAs of nonzero weight    }  \ar[d]  \ar[l]             \\
                    &\txt{a class of factorizable \\Lie conformal bialgebras}  \ar[u]\ar@<-1ex>[r]
                    &\txt{a class of  quadratic    Rota-Baxter\\ Lie conformal algebras of the same weight } \ar[u] \ar[l]  }               }}
    \end{equation*}
In particular, a GDBA with a Rota-Baxter operator of nonzero weight
gives a quadratic Rota-Baxter GDA as well as a Lie conformal
bialgebra with a Rota-Baxter operator, which both give rise to the
same factorizable Lie conformal bialgebra.

Additionally, as the combination of  above two approaches, for
factorizable GDBAs and their corresponding quadratic Rota-Baxter
GDAs of nonzero weight, we also consider the induced structures
and their corresponding conformal structures. We illustrate these
correspondences between these induced structures from factorizable
Novikov bialgebras, which are a special case of factorizable
GDBAs, and those induced from the corresponding class of
factorizable Lie conformal bialgebras.

    The paper is organized as follows. In Section \ref{2}, we recall some basic facts on Lie conformal algebras as well as Lie conformal bialgebras.
    In Section \ref{3}, we introduce the notion of factorizable Lie conformal bialgebras and establish a one-to-one correspondence with quadratic Rota-Baxter Lie conformal algebras of nonzero weight.
    We construct a quadratic Rota-Baxter Lie conformal algebra from a Lie conformal bialgebra with a Rota-Baxter operator of nonzero weight, which produces a factorizable Lie conformal bialgebra.
    We also show that factorizable Lie conformal bialgebras induce generalized pseudo-Hessian post-Lie conformal algebras and special partial-pre-post-Lie conformal algebras.
    In Section \ref{4},  we  introduce the notion of factorizable GDBAs and establish their correspondence with a  class of factorizable Lie conformal bialgebras.
    There is also a correspondence between quadratic  Rota-Baxter GDAs of nonzero weight and a class of quadratic Rota-Baxter Lie conformal algebras of the same weight.
In particular, there is a construction of factorizable GDBAs and
hence factorizable Lie conformal bialgebras from GDBAs with a
Rota-Baxter operator of nonzero weight. Moreover, factorizable
Novikov bialgebras induce post-Novikov algebras and special
partial-pre-post-Novikov algebras which correspond respectively to
certain classes of post-Lie conformal algebras and
partial-pre-post-Lie conformal algebras.

    Throughout this paper, let  $\bf k$ be a field of characteristic zero. All tensors over ${\bf k}$ are denoted by $\otimes$. We denote the identity map by $\id$. All vector spaces and algebras over ${\bf k}$ (resp. all ${\bf k}[\partial]$-modules and Lie conformal algebras) are assumed to be finite-dimensional (resp. finite) unless otherwise stated, even though many results still hold in the infinite-dimensional (resp. infinite) cases. Moreover, if $A$ is a vector space, then we denote the space of polynomials of $\lambda$ with coefficients in $A$ by $A[\lambda]$. For a vector space or a ${\bf k}[\partial]$-module $A$, let
    \begin{equation*}
        \tau:A\otimes A\rightarrow A\otimes A,\qquad a\otimes b\mapsto b\otimes a,\qquad \forall a,b\in A,
    \end{equation*}
    be the flip operator. Then we denote $\tau_{12} = \tau \otimes \id$, $ \tau_{23} = \id \otimes \tau$, $\tau_{13} = (\tau\otimes \id)(\id \otimes \tau)(\tau \otimes \id)$.

    \section{Some basic facts on Lie conformal bialgebras}\label{2}
    We  recall some basic facts on Lie conformal algebras and Lie conformal bialgebras.


    \begin{definition}\cite{K1}\label{lca}
        A {\bf Lie conformal algebra} $(\A,[\cdot_\lambda\cdot])$ is a ${\bf k}[\partial]$-module $\A$ with a $\lambda$-bracket $[\cdot_\lambda\cdot]:\A\times \A\rightarrow \A[\lambda]$ which is a
        ${\bf k}$-bilinear map  satisfying
        \begin{align}
            &[\partial a_\lambda b]=-\lambda[a_\lambda b],   \quad\quad\quad\quad   [ a_\lambda \partial b]=(\lambda+\partial)[a_\lambda b],\tag{conformal sesquilinearity} \label{cs}\\
            &[a_\lambda b]=-[b_{-\lambda-\partial}a],\tag{skew-symmetry}\\
            &[a_\lambda[b_\mu c]]=[[a_\lambda b]_{\lambda+\mu}c]+[b_\mu[a_\lambda c]], \qquad \forall a,b,c\in \A.  \tag{Jacobi identity}   \label{li}
        \end{align}
        A Lie conformal algebra $(\A,[\cdot_\lambda\cdot])$ is called {\bf finite}, if it is finitely generated as a ${\bf k}[\partial]$-module. Otherwise, we call it {\bf infinite}.
    \end{definition}

    \begin{definition}\cite{K1}
        Let $U$ and $V$ be two ${\bf k}[\partial]$-modules. A {\bf conformal linear map} from $U$ to $V$ is a ${\bf k}$-linear map $f:U\rightarrow V[\lambda]$, denoted by $f_\lambda:U\rightarrow V$,   such that  $f_\lambda(\partial u)=(\lambda+\partial) f_\lambda u$, for all $u\in U$.
        The vector space of all such conformal linear maps, denoted by $\text{Chom}(U,V)$, has a $ {\bf k}[\partial]$-module structure given by
        \begin{equation*}
            (\partial f)_\lambda u=-\lambda f_\lambda u,\qquad \forall f\in \text{Chom}(U,V).
        \end{equation*}
        In particular, if $U=V$, then we denote $\text{Chom}(V,V)$ by  $\text{Cend}(V)$. If $V$ is a finitely generated  $ {\bf k}[\partial]$-module, then  $\text{Cend}(V)$ has a Lie conformal  algebra structure defined by
        $$
        [f_\lambda g]_\mu v=f_\lambda(g_{\mu-\lambda}v)-g_{\mu-\lambda}(f_\lambda v),\qquad \forall f,g\in \text{Cend}(V),v\in V.
        $$
        In this case, we denote $(\text{Cend}(V),[\cdot_\lambda\cdot])$ by $\text{gc}(V)$, which is called the {\bf general Lie conformal algebra}  on $V$.
        We define the {\bf conformal dual} of a $ {\bf k}[\partial]$-module $U$ as $ U^{*c}=\text{Chom}(U,{\bf k})$, where ${\bf k}$ is viewed as the trivial ${\bf k}[\partial]$-module, that is,
        $U^{*c}=\{f: U\rightarrow {\bf k}\;| \text{ ${\bf k}$-linear and }f_\lambda(\partial b)=\lambda f_\lambda b \;\;\text{for all $b\in U$}\} $.

        Let $U$, $V$ and $W$ be ${\bf k} [\partial]$-modules.   A {\bf conformal bilinear map}   $f_\lambda(\cdot,\cdot): U\times V\rightarrow W[\lambda]$  is a ${\bf k}$-bilinear map satisfying that $f_\lambda(\partial u,v)=-\lambda f_\lambda(u,v)$ and $f_\lambda( u,\partial v)=(\lambda+\partial) f_\lambda(u,v) $ for all $u\in U$ and $v\in V$.
    \end{definition}

    For convenience, in the sequel, we define
    \begin{eqnarray*}
        \langle x,a\rangle_\lambda=\langle a, x\rangle_{-\lambda}=x_\lambda a,\qquad
        \langle x\otimes y, a\otimes b\rangle_{\lambda, \mu}=\langle x, a\rangle_\lambda \langle y,b\rangle_\mu,\;\;\forall a, b\in \A,\;\;x, y\in \A^{\ast c}.
    \end{eqnarray*}

    \begin{definition}\cite{K1}
        Let $M$ be a finitely generated ${\bf k}[\partial]$-module and $(\A,[\cdot_\lambda \cdot])$ be a Lie conformal algebra. $(M, \rho)$ is called a {\bf  representation} of $(\A,[\cdot_\lambda \cdot])$ if $\rho: \A\rightarrow \text{gc}(M)$ is a homomorphism of Lie conformal algebras.
    \end{definition}

    \begin{proposition}\cite{L}\label{prop-dual}
        Let $(\A,[\cdot_\lambda \cdot])$ be a Lie conformal algebra and $(M, \rho)$ be a representation of $(\A,[\cdot_\lambda \cdot])$, where $M$ is free as a ${\bf k}[\partial]$-module. Let $\rho^*$ be a $\mathbf{k}[\partial]$-module homomorphism from $\A$ to $\mathrm{Cend}~(M^{*c})$ given by
        \begin{eqnarray*}
            (\rho^\ast(a)_\lambda \varphi)_\mu u=-\varphi_{\mu-\lambda}(\rho(a)_\lambda u),\;\;\;\forall a\in \A, \varphi\in M^{\ast c}, u\in M.
        \end{eqnarray*}
        Then   $(M^{\ast c}, \rho^\ast)$ is a representation of $(\A,[ \cdot_\lambda\cdot])$, which is call the {\bf  dual representation}  of $(\A,[ \cdot_\lambda\cdot])$.
         \end{proposition}
        \begin{example}
        	Let $(\A,[\cdot_\lambda \cdot])$ be a Lie conformal algebra where $\A$ is free as a ${\bf k}[\partial]$-module.
        Define $\mathrm{ad}:\A\rightarrow \text{gc}(\A)$ by $\mathrm{ad}(a)_\lambda b=[a_\lambda b]$ for all $a,b\in \A$. Then $(\A,\mathrm{ad})$ is a representation of $(\A,[\cdot_\lambda \cdot])$, called the {\bf adjoint representation}. Moreover, $(\A^{*c},\mathrm{ad}^*)$ is also a representation of $(\A,[\cdot_\lambda \cdot])$, called the {\bf coadjoint representation}. If there is a Lie conformal algebra structure on $\A^{\ast c}$, then we denote the adjoint representation of $(\A^{\ast c}, [\cdot_\lambda \cdot])$ by $(\A^{\ast c}, \ad)$.

        On the other hand, by \cite[Proposition 3.1]{HL1}, $(M,\rho)$ is a representation of a Lie conformal algebra $(\A,[\cdot_\lambda\cdot])$ if and only if there is a Lie conformal algebra structure on the direct sum $\A\oplus M$ as ${\bf k}[\partial]$-modules given by
        \begin{equation}\label{eq-rep}
            [(a,x)_\lambda (b,y)]=([a_\lambda b],\rho(a)_\lambda y-\rho(b)_{-\lambda-\partial}x),\qquad \forall a,b\in\A,x,y\in M.
        \end{equation}
        We denote this Lie conformal algebra by $\A\ltimes_{\rho} M$, which is called the {\bf semi-direct product of $(\A,[\cdot_\lambda\cdot])$ and its representation $(M, \rho)$}.

\end{example}
    Set
   \begin{eqnarray*}
       \partial^{\otimes^2}:=\partial \otimes \id+\id\otimes \partial,\;\;\partial^{\otimes^3}:=\partial\otimes \id\otimes \id+\id\otimes \partial\otimes \id+\id\otimes \id\otimes
     \partial.
    \end{eqnarray*}

    \begin{definition} \cite{L}
        A {\bf Lie conformal coalgebra} $(\A,\Delta)$ is a ${\bf k}[\partial]$-module $\A$  endowed with a ${\bf k}[\partial]$-module homomorphism $\Delta: \A\rightarrow \A\otimes \A$ such that
        \begin{align*}
            \Delta(a)=-\tau\Delta(a),\qquad (\id\otimes \Delta)\Delta(a)-(\tau\otimes \id)(\id\otimes \Delta)\Delta(a)=(\Delta\otimes \id)\Delta(a),\qquad   \forall a\in \A,
        \end{align*}
        where the ${\bf k}[\partial] $-module structure on $\A\otimes \A$ is defined by $\partial(a\otimes b)=(\partial a)\otimes b+a\otimes (\partial b)$ for all $a$, $b\in \A$.

        A {\bf Lie conformal bialgebra}  is a triple $(\A,[\cdot_\lambda \cdot],\Delta)$ such that $(\A,[\cdot_\lambda \cdot])$ is a Lie conformal algebra, $(\A,\Delta)$ is a Lie conformal coalgebra, and they satisfy the following condition.
        \begin{align*}
            \Delta([a_\lambda b])&=      a_\lambda\Delta(b)-b_{-\lambda-\partial}\Delta(a)\\
            &= [a_\lambda b_{(1)}]\otimes b_{(2)}+ b_{(1)}\otimes [a_\lambda b_{(2)}]-[b_{-\lambda-\partial^{\otimes^2}} a_{(1)}]\otimes a_{(2)}- a_{(1)}\otimes [b_{-\lambda-\partial^{\otimes^2}} a_{(2)}],~~~ \forall a,b\in \A,
        \end{align*}
        where we use the Sweedler notation $\Delta(c)=c_{(1)}\otimes c_{(2)}$ (with summation implied) for all $c\in \A$.
       
    \end{definition}

    Recall \cite{HL1} that a {\bf matched pair of Lie conformal algebras}  consists of two  Lie conformal algebras $(\A_1,[\cdot_\lambda \cdot]^1)$ and $(\A_2,[\cdot_\lambda \cdot]^2)$,  together with a representation $(\A_2, \rho)$ of $(\A_1,[\cdot_\lambda \cdot]^1)$ and a representation $(\A_1, \sigma)$ of $(\A_2,[\cdot_\lambda \cdot]^2)$ such that
    \begin{align*}
        &\rho(a)_\lambda[x_\mu y]^2=[(\rho(a)_\lambda x)_{\lambda+\mu }y]^2+[x_\mu(\rho(a)_\lambda y)]^2-\rho(\sigma(x)_{-\lambda-\partial }a)_{\lambda+\mu}y+\rho(\sigma(y)_{-\lambda-\partial}a)_{-\mu-\partial}x,\\
        &\sigma(x)_{-\lambda-\mu-\partial}[a_\lambda b]^1=[a_\lambda(\sigma(x)_{-\mu-\partial}b)]^1-[b_\mu(\sigma(x)_{-\lambda-\partial} a)]^1-\sigma(\rho(a)_{\lambda }x)_{-\mu-\partial}b+\sigma(\rho(b)_{\mu}x)_{-\lambda-\partial}a,
    \end{align*}
    for all $a,b\in \A_1$ and $x,y\in \A_2$. Denote it by $(\A_1,\A_2, \rho,\sigma)$. Note that $(\A_1,\A_2, \rho,\sigma)$ is a matched pair of Lie conformal algebras if and only if there is a Lie conformal algebra structure on the direct sum $\A_1\oplus \A_2$ as ${\bf k}[\partial]$-modules given by
    \begin{equation}\label{mplb}
        [(a,x)_\lambda (b,y)]^\Join=([a_\lambda b]^1+\sigma(x)_\lambda b-\sigma(y)_{-\lambda-\partial}a,[x_\lambda y]^2+\rho(a)_\lambda y-\rho(b)_{-\lambda-\partial}x)
    \end{equation}
    for all $a,b\in \A_1$ and $x,y\in \A_2$. Denote this Lie conformal algebra by $\A_1\Join \A_2$.

    \begin{definition}\cite{BKL}
        Let $M$ be a ${\bf k}[\partial]$-module. A {\bf conformal bilinear form} on $M$  is a conformal bilinear map $\omega_\lambda(\cdot,\cdot):M\times M\rightarrow {\bf k}[\lambda]$, where ${\bf k}$ is seen as the trivial ${\bf k}[\partial]$-module. A conformal bilinear form $\omega_\lambda(\cdot,\cdot)$  on $M$ is called {\bf symmetric } if
        \begin{equation*}
            \omega_\lambda(a,b)=\omega_{-\lambda}(b,a),\qquad  \forall  a,b\in M.
        \end{equation*}
        A conformal bilinear form $\omega_\lambda(\cdot,\cdot)$ on a Lie conformal algebra $(\A,[\cdot_\lambda\cdot])$ is called {\bf invariant} if
        \begin{equation}\label{cbf1}
            \omega_\lambda([a_\mu b],c)=\omega_{\mu}(a,[b_{\lambda-\mu}c]),\qquad \forall a,b,c\in \A.
        \end{equation}
    \end{definition}

    Recall \cite{L} that a  {\bf conformal Manin triple} of Lie conformal algebras is a triple $(\A, \A_1,\A_2)$ of Lie conformal algebras  where $\A_1$, $\A_2$ are  Lie conformal subalgebras of $\A$ and $\A = \A_1\oplus \A_2$ as {\bf k}$[\partial]$-modules is equipped with a non-degenerate invariant symmetric conformal bilinear form  $\omega_\lambda(\cdot,\cdot)$ such that  $\A_1$ and $\A_2$ are isotropic with respect to $\omega_\lambda(\cdot,\cdot)$, that is, $\omega_\lambda(\A_i,\A_i)=0$  for $i =  1,2$.

    \begin{proposition}\cite[Theorem 3.9]{L}, \cite[Proposition 4.18]{HBG} \label{LCB-equ}
        Let $(\A, [\cdot_\lambda \cdot])$ be a  Lie conformal algebra and $(\A, \Delta)$ be a Lie conformal coalgebra, where $\A$ is free as a ${\bf k}[\partial]$-module. Suppose that there is a
        Lie conformal algebra  $(\A^{\ast c}, [\cdot_\lambda \cdot]^\ast)$  in which $[\cdot_\lambda \cdot]^\ast
        $ is obtained from $\Delta$ via
      \begin{equation}\label{eq:dual}\langle [x_\mu y]^\ast, a\rangle_\lambda=\langle  x\otimes y, \Delta(a)\rangle_{\mu, \lambda-\mu},\qquad \forall x, y\in \A^{\ast c}, a\in \A.\end{equation} Then the following conditions are equivalent.
        \begin{enumerate}
            \item \mlabel{it:11}$(\A, [\cdot_\lambda \cdot], \Delta)$ is a Lie conformal
            bialgebra.
            \item \mlabel{it:12} $(\A\oplus \A^{\ast c}, \A, \A^{\ast c})$ is a conformal Manin triple of Lie conformal algebras with the conformal bilinear form given by
            \begin{eqnarray}\label{eq-blf}
                \omega_\lambda ( (a,x), (b,y) )=\langle x,b\rangle_\lambda+\langle y,a\rangle_{-\lambda}, \qquad\forall a, b\in \A, x, y \in \A^{\ast c}.
            \end{eqnarray}
            \item \mlabel{it:13} $(\A, \A^{\ast c}, {\rm ad}^\ast, \ad^\ast)$ is a matched pair of Lie conformal algebras.
        \end{enumerate}
    \end{proposition}

    \begin{definition} Two Lie conformal bialgebras $(\A ,[\cdot_\lambda\cdot],\Delta)$ and $(\A ',[\cdot_\lambda\cdot]',\Delta')$
        are called  {\bf isomorphic} if there exists a Lie conformal algebra isomorphism $\phi: (\A ,[\cdot_\lambda\cdot])\rightarrow (\A ',[\cdot_\lambda\cdot]')$
        satisfying
        \begin{equation*}
            (\phi\otimes\phi)(\Delta(a))=\Delta'(\phi(a)),\qquad  \forall a\in \A.
        \end{equation*}
    \end{definition}
    \begin{remark}
       By Proposition \ref{LCB-equ}, if a  Lie conformal bialgebra $(\A,[\cdot_\lambda\cdot],\Delta)$ is free as ${\bf k}[\partial]$-module, then  $(\A\oplus \A^{\ast c}, \A, \A^{\ast c})$ is a conformal Manin triple of Lie conformal algebras. Thus we also denote such a Lie conformal bialgebra   by $(\A,\A^{*c})$.
        Let $(\A,[\cdot_\lambda \cdot], \Delta)$ and $(\A '$, $[\cdot_\lambda\cdot]'$, $\Delta')$ be two Lie conformal bialgebras, where $\A$ and $\A^{'}$ are free as ${\bf k}[\partial]$-modules.   A ${\bf k}[\partial]$-module isomorphism $\phi: (\A,[\cdot_\lambda \cdot], \Delta) \rightarrow (\A ',[\cdot_\lambda\cdot]',\Delta')$ is an isomorphism of Lie conformal bialgebras if and only if $\phi: (\A, [\cdot_\lambda \cdot])\rightarrow (\A ',[\cdot_\lambda\cdot]')$ and $\phi^\ast: ({\A'}^{\ast c}, {[\cdot_\lambda \cdot]'}^\ast)\rightarrow ({\A}^{\ast c}, {[\cdot_\lambda \cdot]}^\ast)$ are isomorphisms of Lie conformal algebras, where $\langle \phi^\ast(x), a\rangle_\lambda=\langle x, \phi(a)\rangle_\lambda$ for all $x\in {\A'}^{\ast c}$ and $a\in \A$.
    \end{remark}

    Recall the notion of the conformal classical Yang-Baxter equation.
    \begin{definition}\cite{L}
        Let $(\A,[\cdot_\lambda\cdot])$ be a Lie conformal algebra and $r=\sum_i a_i\otimes b_i\in \A\otimes \A$.   The {\bf conformal classical Yang-Baxter equation (CCYBE) in $(\A,[\cdot_\lambda\cdot])$} is
        the following equation
        \begin{eqnarray*}
            [\![r,r]\!]:&=& \sum_{i,j}([{a_i}_\mu a_j]\otimes b_i\otimes b_j|_{\mu=\id\otimes \partial\otimes \id}\\
            &&-a_i\otimes [{a_j}_\mu b_i]\otimes b_j|_{\mu=\id\otimes \id\otimes \partial}-a_i\otimes a_j\otimes [{b_j}_\mu b_i]|_{\mu=\id\otimes \partial\otimes \id})\\
            &\equiv &0 \;\; (\text{mod} ~~\partial \cdot (\A\otimes \A\otimes \A)),
        \end{eqnarray*}
        i.e., $[\![r,r]\!]$ belongs to the image of $\partial$ under the module action of ${\bf k}[\partial]$ on $\A\otimes \A\otimes \A$ given by $ \partial \cdot (a\otimes b\otimes c)=\partial^{\otimes^3}(a\otimes b\otimes c)$ for all $a$, $b$, $c\in \A$.

\end{definition}
 
\begin{definition}
    \cite{L}
    Let $(\A, [\cdot_\lambda \cdot])$ be a Lie conformal algebra and $r=\sum_i a_i\otimes b_i\in \A\otimes \A$. If $(\A, [\cdot_\lambda \cdot], \Delta_r)$ is a Lie conformal bialgebra, where $\Delta_r$ is defined by
    \begin{eqnarray}\mlabel{coboundary}
        \Delta_r(a)=a_{-\partial^{\otimes^2}} r=\sum_i ([a_{-\partial^{\otimes^2}} a_i]\otimes b_i+a_i\otimes [a_{-\partial^{\otimes^2}} b_i]),\qquad\forall a\in \A,
    \end{eqnarray}
    then  $(\A,[\cdot_\lambda\cdot],\Delta_r)$ is called a {\bf coboundary Lie conformal bialgebra}.
\end{definition}
\begin{remark}\label{rmk-11}
Let $r_1$, $r_2\in \A\otimes \A$. By the CCYBE in
$(\A,[\cdot_\lambda\cdot])$ and Eq. (\ref{coboundary}), it is
straightforward to show that if $r_1\equiv
r_2~~\text{mod}~~(\partial\otimes \id+\id\otimes
\partial)(\A\otimes \A)$, then $[\![r_1,r_1]\!]\equiv[\![r_2,r_2]\!]\quad
(\mathrm{mod}~\partial\cdot(\A\otimes \A\otimes \A))$ and
$\Delta_{r_1}=\Delta_{r_2}$.
\end{remark}

 For a Lie conformal algebra $(\A, [\cdot_\lambda \cdot])$,   $r\in \A\otimes
 \A$ is called {\bf $\A$-invariant} if
    \begin{equation*}
        a_{-\partial^{\otimes^2}}(r)=0,\qquad \forall a\in \A.
    \end{equation*}

\begin{proposition}\cite[Theorem 3.4]{L}\label{prop-cob}
    Let $(\A, [\cdot_\lambda \cdot])$ be a Lie conformal algebra and $r\in \A\otimes \A$. Then $(\A, [\cdot_\lambda \cdot], \Delta_r)$ is a coboundary Lie conformal bialgebra if and only if
    $r$ satisfies
    $$a_\lambda [\![r,r]\!]|_{\lambda=\partial^{\otimes^3}}=0,\qquad\forall a\in \A,$$
    and
    the symmetric part of $r$ is {$\A$-invariant}, where  the action of $\A$  on $\A\otimes \A\otimes \A$  is given by
    $$
    a_\lambda (b\otimes c\otimes d)=[a_\lambda b]\otimes c\otimes d +b\otimes [a_\lambda c]\otimes d+b\otimes c\otimes [a_\lambda d],
    \qquad \forall a,b,c,d\in A.
    $$

\end{proposition}

\begin{pdef}
    \cite{L} Let $(\A, [\cdot_\lambda \cdot])$ be a Lie conformal algebra and $r\in \A\otimes \A$. If $r$ is a solution of the CCYBE in $(\A, [\cdot_\lambda \cdot])$ and the
     symmetric part of $r$ is $\A$-invariant, then $(\A, [\cdot_\lambda\cdot], \Delta_r)$ is a Lie conformal bialgebra, which is called a {\bf quasi-triangular Lie conformal bialgebra},
     where $\Delta_r$ is given by Eq. (\ref{coboundary}). In particular, if $r$ is a skew-symmetric solution of the CCYBE in $(\A, [\cdot_\lambda \cdot])$, then $(\A, [\cdot_\lambda\cdot], \Delta_r)$ is called a {\bf triangular Lie conformal bialgebra}.

\end{pdef}

\section{Factorizable Lie conformal bialgebras via quadratic Rota-Baxter Lie conformal algebras of nonzero weight and some induced structures}\label{3}

We first introduce the notion of factorizable Lie conformal
bialgebras and establish their correspondence with quadratic
Rota-Baxter Lie conformal algebras of nonzero weight. Building on
this correspondence, we then construct factorizable Lie conformal
bialgebras from  Lie conformal bialgebras with a Rota-Baxter
operator of nonzero weight. Finally, we show that generalized
pseudo-Hessian post-Lie conformal algebras and special
partial-pre-post-Lie conformal algebras are induced from
factorizable Lie conformal bialgebras.

\subsection{Factorizable Lie conformal bialgebras}

    Let $(\A, [\cdot_\lambda \cdot])$ be a Lie conformal algebra, where $\A$ is free as a ${\bf k}[\partial]$-module, and $r=\sum_i a_i\otimes b_i\in \A\otimes \A$.    Define two ${\bf k}[\partial]$-module homomorphisms $r_+$, $r_-: \A^{*c}\rightarrow \A$ by
    \begin{equation}
\label{Eq-r}    r_+(x)=\sum_i \langle x,a_i\rangle_{-\partial} b_i,\qquad   r_-(x)=-\sum_i \langle x,b_i\rangle_{-\partial} a_i,\qquad \forall x\in \A^{*c}.
\end{equation}

\begin{remark}\label{rmk-2}
Let $r_1$, $r_2\in \A\otimes \A$. If $r_1\equiv r_2
~~\text{mod}~~(\partial\otimes \id+\id\otimes
\partial)(\A\otimes\A)$, then it is straightforward to show that
${r_1}_+={r_2}_+$ and ${r_1}_-={r_2}_-$.

Let $\{e_1, \ldots, e_n\}$ be a ${\bf k}[\partial]$-basis of $\A$
and $\{e_1^\ast, \ldots, e_n^\ast\}$ be the dual ${\bf
k}[\partial]$-basis of $\A^{\ast c}$. 
Set
$$r=\sum_{i=1}^n\sum_{j=1}^nf_{ij}(\partial\otimes \id,\id\otimes
\partial)e_i\otimes e_j.$$ Then for any $k\in \{1, \ldots, n\}$,
we obtain
\begin{eqnarray}
r_+(e_k^\ast)=\sum_{j=1}^nf_{kj}(-\partial, \partial)e_j,\;\;\;r_-(e_k^\ast)=-\sum_{i=1}^nf_{ik}(\partial, -\partial)e_i.
\end{eqnarray}
Therefore, $r$ $\text{mod}~~(\partial\otimes \id+\id\otimes \partial)(\A\otimes\A)$ is determined by $r_+$ or $r_-$.

By the discussion above, if
$r_+(e_k^\ast)=\sum_{j=1}^ng_{kj}(\partial)e_j$ for $k\in \{1,
\ldots, n\}$ where $g_{kj}(\partial)\in {\bf k}[\partial]$, then
we get $r\equiv \sum_{k=1}^n\sum_{j=1}^ng_{kj}(\id\otimes
\partial)e_k\otimes e_j$ $\text{mod}~~(\partial\otimes
\id+\id\otimes \partial)(\A\otimes\A)$. By Remark \ref{rmk-11}, in
the study of coboundary Lie conformal bialgebras, we choose
$$r=\sum_{k=1}^n\sum_{j=1}^ng_{kj}(\id\otimes \partial)e_k\otimes
e_j=\sum_{k=1}^n e_k\otimes r_+(e_k^*).$$
\end{remark}
 
\begin{proposition}\label{prop-1}
With the assumption above, define a $\lambda$-bracket on $\A^{*c}$ by
\begin{equation}\label{acl}
    [x_\lambda y]^r={\rm ad}^*(r_+(x))_\lambda y-{\rm ad}^*(r_-(y))_{-\lambda-\partial}x,\qquad   \forall x,y\in \A^{*c}.
\end{equation}
Then $(\A^{*c},[\cdot_\lambda\cdot]^r)$ is a Lie conformal algebra if and only if $(\A, \Delta_r)$ is a Lie conformal coalgebra, where $\Delta_r$ is given by Eq. (\ref{coboundary}).
\end{proposition}
\begin{proof}
For all $x$, $y\in \A^{*c}$, $a\in \A$ and $r=\sum_i a_i\otimes b_i\in\A\otimes \A$, we have
\begin{eqnarray*}
    \langle [x_\lambda y]^r,a \rangle_\mu&=&\langle {\rm ad}^*(r_+(x))_\lambda y-{\rm ad}^*(r_-(y))_{-\lambda-\partial}x,a\rangle_\mu\\
    &=&-\langle y, [{r_+(x)}_\lambda a]\rangle_{\mu-\lambda}+\langle x, [{r_-(y)}_{-\lambda+\mu} a]\rangle_\lambda\\
    &=&\sum_i\langle  x\otimes y,-a_i\otimes [{b_i}_{\lambda}a ]-[{a_i}_{\mu-\lambda}a]\otimes b_i\rangle_{\lambda,\mu-\lambda}\\
    &=&\sum_i\langle  x\otimes y,a_i\otimes [a_{-\mu}b_i]+[a_{-\mu}a_i]\otimes b_i\rangle_{\lambda,\mu-\lambda}\\
    &=&\langle x\otimes y,\Delta_r(a)\rangle_{\lambda,\mu-\lambda}.
\end{eqnarray*}
Then this conclusion follows directly by
\cite[Proposition 2.13]{L}.
\end{proof}

Let $U$ and $W$ be two ${\bf k}[\partial]$-modules and $T: U\rightarrow W$ be a ${\bf k}[\partial]$-module homomorphism. Define a ${\bf k}[\partial]$-module homomorphism $T^\ast: W^{\ast c}\rightarrow U^{\ast c}$ by
\begin{eqnarray*}
\langle T^{\ast}(f),u\rangle_\lambda=\langle f, T(u)\rangle_\lambda,\qquad  \forall f\in  W^{\ast c},\;\;u\in U.
\end{eqnarray*}

Set
\begin{equation*}
I=r_+-r_-:\A^{*c}\rightarrow \A.
\end{equation*}
Note that $I$ is a ${\bf k}[\partial]$-module homomorphism.  Obviously, if $r$ is skew-symmetric, then $I=0$. 
\begin{lemma}\label{I}
Let $ (\A,[\cdot_\lambda\cdot])$ be a Lie conformal algebra, where
$\A$ is free as a ${\bf k}[\partial]$-module, and $r\in \A\otimes
\A$. Then the following conclusions hold.
\begin{enumerate}
    \item \label{item1}$I=I^*$. In particular, if $I$ is an isomorphism of Lie conformal algebras, then $(I^{-1})^*=(I^*)^{-1}=I^{-1}$.
    \item \label{item2}The symmetric part of $r$ is $\A$-invariant if and only if
   \begin{equation}\label{eq:rr1}
   	 I\circ {\rm ad}^*(a)_\lambda={\rm ad}(a)_\lambda\circ I,\qquad \forall a\in \A.
   	 \end{equation}
 
\end{enumerate}
\end{lemma}
\begin{proof}
Let $r=\sum_i a_i \otimes b_i$.

(\ref{item1})
   Let $x$, $y\in \A^{*c}$. Then we obtain
    \begin{eqnarray*}
        \langle I^*(x),y\rangle_\lambda&=&\langle x,I(y)\rangle_\lambda=\sum_i\langle x,\langle y,a_i\rangle_{-\partial}b_i+\langle y,b_i\rangle_{-\partial}a_i\rangle_\lambda\\
        &=&\sum_i(\langle x,b_i\rangle_\lambda\langle y,a_i\rangle_{-\lambda}+\langle x,a_i\rangle_\lambda\langle y,b_i\rangle_{-\lambda})\\
        &=&\langle r_+(x)-r_-(x),y\rangle_\lambda=\langle
        I(x),y\rangle_\lambda.
    \end{eqnarray*}
    Hence $I=I^*$. Moreover, if $I$ is an isomorphism of Lie conformal algebras, then we get
    \begin{eqnarray*}
        \langle x,a \rangle_\lambda =\langle x,(I\circ I^{-1}) (a)\rangle_\lambda =\langle I^*(x),I^{-1}(a ) \rangle_\lambda=\langle ((I^{-1})^*\circ I^*)(x),a\rangle_\lambda.
    \end{eqnarray*}
 Thus $({I^\ast})^{-1}=(I^{-1})^*$. 

 (\ref{item2})   Let $x,y\in \A^{*c}$ and $a\in \A$.  By Item (\ref{item1}), we have
    \begin{eqnarray*}
        &&\sum_i   \langle x\otimes y, a_i\otimes [a_{-\partial^{\otimes^2}}b_i]+ [a_{-\partial^{\otimes^2}}a_i]\otimes  b_i  + [a_{-\partial^{\otimes^2}}b_i]\otimes a_i+ b_i \otimes [a_{-\partial^{\otimes^2}}a_i]  \rangle_{ \lambda,\mu}\\
        &&\quad=        \langle y, \sum_i ( \langle {\rm ad}^*(a)_{-\lambda-\mu}x, a_i\rangle_{-\mu}b_i +\langle {\rm ad}^*(a)_{-\lambda-\mu}x, b_i\rangle_{-\mu}a_i )  \rangle_\mu  \\
        &&\quad\quad+\langle x, \sum_i ( \langle {\rm ad}^*(a)_{-\lambda-\mu}y, b_i\rangle_{-\lambda}a_i +\langle {\rm ad}^*(a)_{-\lambda-\mu}y, a_i\rangle_{-\lambda}b_i)  \rangle_\lambda   \\
        &&\quad=\langle y, (r_+-r_-)({\rm ad}^*(a)_{-\lambda-\mu}x )\rangle_\mu
        +\langle x, (r_+-r_-)({\rm ad}^*(a)_{-\mu-\lambda}y )\rangle_\lambda\\
        &&\quad=\langle y, I ({\rm ad}^*(a)_{-\lambda-\mu}x )\rangle_\mu
        +\langle   ({\rm ad}^*(a)_{-\mu-\lambda}y ),I(x)  \rangle_{-\lambda}\\
        &&\quad=\langle y, I ({\rm ad}^*(a)_{-\lambda-\partial}x )\rangle_\mu
        - \langle   y ,({\rm ad}(a)_{-\lambda-\partial})I(x)
        \rangle_{\mu}.
    \end{eqnarray*}
    Hence the symmetric part of $r$ is $\A$-invariant if and only
    if Eq.~(\ref{eq:rr1}) holds.
\end{proof}

\begin{remark}\label{rmk-d}
With the assumption in Lemma \ref{I}, we obtain
\begin{eqnarray*}
    \langle r_-^\ast(x),y\rangle_\lambda&=&\langle x, r_-(y)\rangle_\lambda=\langle x, -\sum_i\langle y,b_i\rangle_{-\partial} a_i\rangle_\lambda\\
    &=&-\sum_i\langle y,b_i\rangle_{-\lambda} \langle x, a_i\rangle_\lambda=-\langle r_+(x),y\rangle_\lambda,\;\;\;\forall x, y\in \A^{\ast^c}.
\end{eqnarray*}
Therefore, we get $r_-^\ast=-r_+$. Similarly, we have $r_-=-r_+^\ast$.
 \end{remark}
\begin{theorem}\label{qtl}
Let $(\A,[\cdot_\lambda\cdot])$ be a  Lie conformal algebra, where
$\A$ is free as a ${\bf k}[\partial]$-module, and $r\in \A\otimes
\A$. If the symmetric part of  $r$ is $\A$-invariant, then $r$ is
a solution of the CCYBE in $(\A,[\cdot_\lambda\cdot])$ if and only
if $(\A^{*c},[\cdot_\lambda\cdot]^r)$ is a Lie conformal algebra
where the $\lambda$-bracket is defined by Eq. \eqref{acl} and
$r_-: (\A^{*c},[\cdot_\lambda\cdot]^r) \rightarrow
(\A,[\cdot_\lambda\cdot])$  is a Lie conformal algebra
homomorphism.
\end{theorem}
\begin{proof}
By Proposition \ref{prop-cob}, $(\A, [\cdot_\lambda \cdot], \Delta_r)$ is a coboundary Lie conformal bialgebra if and only if the symmetric part of $r$ is $\A$-invariant and $a_{-\partial^{\otimes^3}}[\![r,r]\!]=0$ for all $a\in \A$. Moreover, $[\![r,r]\!]\equiv 0$ $~~(\text{mod} ~~\partial \cdot (\A\otimes \A \otimes \A))$ implies that $ a_{ -\partial^{\otimes^3}}[\![r,r]\!]=0$.  Thus, if the symmetric part of  $r$ is $\A$-invariant and $r$ is a solution of the CCYBE, then $(\A, \Delta_r)$ is a Lie conformal coalgebra. Therefore, by Proposition \ref{prop-1}, $(\A^{*c},[\cdot_\lambda\cdot]^r)$ is a Lie conformal algebra.

Let $r=\sum_i a_i\otimes b_i\in \A\otimes \A$. For all $x,y,z\in \A^{*c}$, we have
\begin{eqnarray*}
    && \langle x,[r_-(y)_\mu r_-(z)]-r_-({\rm ad}^*(r_+(y))_\mu z-{\rm ad}^*(r_-(z))_{-\mu-\partial}y)\rangle_\lambda\\
    &&\quad=\sum_{i,j}(\langle x,\langle y,b_i\rangle_\mu\langle z,b_j\rangle_{-\mu-\partial}[{a_i}_\mu a_j] \rangle_\lambda    -\langle x,a_i\rangle_\lambda\langle z,\langle y,a_j\rangle_\mu[{b_j}_\mu b_i]\rangle_{-\lambda-\mu}\\
    &&\quad\quad-\langle x,a_i\rangle_\lambda\langle y,\langle z,b_j\rangle_{-\lambda-\mu}[{a_j}_{-\lambda-\mu}b_i]\rangle_\mu  ) \\
    &&\quad=\langle x\otimes y\otimes z,[\![r,r]\!] ~~(\text{mod} ~~\partial \cdot (\A\otimes \A \otimes
    \A))\rangle_{\lambda,\mu,-\lambda-\mu}.
\end{eqnarray*}
Hence $r$ is a solution of the CCYBE in $(\A, [\cdot_\lambda
\cdot])$ if and only if $r_-$ is a Lie conformal algebra
homomorphism from $(\A^{*c}, [\cdot_\lambda \cdot]^r)$ to $(\A,
[\cdot_\lambda \cdot])$. Therefore, this
conclusion holds. 
\end{proof}
\begin{remark}\label{rmk-new}
Under the condition that the symmetric part of $r$ is $\A$-invariant, the ${\bf k}[\partial]$-module homomorphism $r_-$
is a Lie conformal algebra homomorphism  if and only if $r_+$  is a Lie conformal algebra homomorphism. Then $I=r_+-r_-$ is also a Lie conformal algebra homomorphism.
\end{remark}


\begin{definition}
A quasi-triangular Lie conformal bialgebra
$(\A,[\cdot_\lambda\cdot],\Delta_r)$ is called {\bf factorizable}
if $\A$ is free as a ${\bf k}[\partial]$-module and the ${\bf
k}[\partial]$-module homomorphism $I: (\A^{*c}, [\cdot_\lambda
\cdot]^r)\rightarrow (\A, [\cdot_\lambda \cdot])$ is an
isomorphism of Lie conformal algebras.
\end{definition}

Consider the map
\[
\xymatrix@C=3em{  
\A^{*c} \ar[r]^{r_+\oplus r_-} & \A\oplus\A \ar[r]^{\varphi} & \A
}
\]
where $\varphi: \A\oplus \A\rightarrow \A$ is defined by $\varphi(a,b)=a-b$ for all $a$, $b\in \A$. Obviously, $r_+\oplus r_-$ and $\varphi$ are ${\bf k}[\partial]$-module homomorphisms, where the action of $\partial$ on $\A\otimes \A$ is given by
$\partial (a, b)=(\partial a, \partial b)$ for all  $a$, $b\in \A$. Note that $I=\varphi \circ (r_+\oplus r_-)$.

\begin{proposition}\label{ppfa}
Let $(\A,[\cdot_\lambda\cdot],\Delta_r) $ be a factorizable Lie conformal bialgebra. Then $\text{Im}(r_+\oplus r_-)$ is a direct sum of Lie conformal subalgebras, which is isomorphic to the Lie conformal algebra $(\A^{*c},[\cdot_\lambda\cdot]^r)$.   In addition, any $a\in\A$ has a unique decomposition $a=a_+-a_-$ where $(a_+,a_-)\in \text{Im}(r_+\oplus r_-)$.
\end{proposition}
\begin{proof}
By Theorem \ref{qtl} and Remark \ref{rmk-new}, $r_+$ and $r_-$ are two Lie conformal algebra homomorphisms, which means that $\text{Im}(r_+\oplus r_-)$ is a direct sum of Lie conformal subalgebras. Since $I=r_+-r_-$ is a Lie conformal algebra isomorphism, it follows that $\text{Im}(r_+\oplus r_-) $ is isomorphic to $(\A^{*c},[\cdot_\lambda\cdot]^r)$ and $\varphi$ is surjective, which completes the proof.
\end{proof}

\subsection{A correspondence between factorizable Lie conformal bialgebras and quadratic Rota-Baxter Lie conformal algebras of nonzero weight}

    \begin{definition}\cite{HB}
Let $(\A,[\cdot_\lambda\cdot])$ be a Lie conformal algebra and
$\theta\in {\bf k}$. A ${\bf k}[\partial]$-module homomorphism
$\mathfrak{B}:\A\rightarrow \A$ is called a {\bf Rota-Baxter
operator of weight $\theta$}  on $(\A,[\cdot_\lambda\cdot])$ if it
satisfies
\begin{equation*}
    [\mathfrak{B}(a)_\lambda \mathfrak{B}(b)]=\mathfrak{B}([\mathfrak{B}(a)_\lambda b]+[a_\lambda \mathfrak{B}(b)]+\theta[a_\lambda b]  ),\qquad \forall a,b\in \A.
\end{equation*}
Moreover, a {\bf   Rota-Baxter Lie conformal
algebra of weight  $\theta$} is a  Lie conformal algebra $(\A,[\cdot_\lambda\cdot])$
with a Rota-Baxter operator $\mathfrak{B}$ of weight $\theta$,
which is denoted by $(\A,[\cdot_\lambda\cdot],\mathfrak{B})$.

\end{definition}

\begin{lemma}\label{bb8}
Let $\mathfrak{B}$ be a Rota-Baxter operator of weight $\theta$ on
a Lie conformal algebra $(\A, [\cdot_\lambda \cdot])$. Then
$\Wh{\mathfrak{B}}=-\mathfrak{B}-\theta \id $ is also a
Rota-Baxter operator of weight $\theta$ on  $(\A, [\cdot_\lambda
\cdot])$.
\end{lemma}
\begin{proof}
It is straightforward.
\end{proof}

    \begin{proposition}\label{prop-2}
Let $(\A,[\cdot_\lambda\cdot],\mathfrak{B})$ be a
Rota-Baxter Lie conformal algebra of weight  $\theta$. Then there is
a new Lie conformal algebra
$(\A,[\cdot_\lambda\cdot]^\mathfrak{B})$, where
$[\cdot_\lambda\cdot]^\mathfrak{B}$ is defined by
\begin{eqnarray*}
    [a_\lambda b]^\mathfrak{B}:=[\mathfrak{B}(a)_\lambda b]+[a_\lambda \mathfrak{B}(b)]+\theta [a_\lambda b],\qquad \forall a,b\in \A.
\end{eqnarray*}
We call $(\A,[\cdot_\lambda \cdot ]^\mathfrak{B})$ the {\bf
descendent Lie conformal algebra} of
$(\A,[\cdot_\lambda\cdot],\mathfrak{B})$ and denote it by
$\A_\mathfrak{B}$. Moreover, $\mathfrak{B}$ is a homomorphism of
Lie conformal algebras from $\A_\mathfrak{B}$ to $(\A,
[\cdot_\lambda \cdot])$ and $\mathfrak{B}$ is still a Rota-Baxter
operator of weight $\theta$ on $\A_\mathfrak{B}$.

\end{proposition} 
\begin{proof}
It is straightforward.
\end{proof}

    \begin{definition}
A {\bf representation} of a    Rota-Baxter Lie
conformal algebra $(\A,[\cdot_\lambda\cdot],$ $\B)$ of weight  $\theta$ is a triple
$(M,\rho,\T)$, where $(M,\rho)$ is a representation of
$(\A,[\cdot_\lambda \cdot])$ and $\T:M\rightarrow M$ is a ${\bf
k}[\partial]$-module homomorphism satisfying
\begin{equation*}
    \rho(\B (a))_\lambda (\T (v))=\T(\rho (\B(a))_\lambda v+\rho(a)_\lambda \T (v)+\theta\rho(a)_\lambda v),\qquad
    \forall a\in \A, v\in M.
\end{equation*}

\end{definition}

\begin{proposition}\label{pp-sprb}
Let $(\A,[\cdot_\lambda\cdot], \B)$ be a
Rota-Baxter Lie conformal algebra of weight  $\theta$, $M$ be a ${\bf
k}[\partial]$-module, and $\T:M\rightarrow M$ be a ${\bf
k}[\partial]$-module homomorphism. Define a ${\bf
k}[\partial]$-module homomorphism $\B\oplus\T$ on the direct sum
$\A\oplus M$ (as ${\bf k}[\partial]$-modules)
 by
$$(\B\oplus\T)(a,x)=(\B (a),\T (x)),\;\; \forall a\in \A, x\in
M.$$ Then $(\A\oplus M,[\cdot_\lambda \cdot], \B\oplus \T)$ is a
   Rota-Baxter Lie conformal algebra of weight  $\theta$ with the
$\lambda$-bracket defined by Eq. (\ref{eq-rep})  if and only if
$(M,\rho,\T) $ is a representation of
$(\A,[\cdot_\lambda\cdot],\B)$.
\end{proposition}

\begin{proof}
It is straightforward.
\end{proof}

\begin{lemma} \label{m8}
Let  $(\A,[\cdot_\lambda \cdot],\B)$ be a
Rota-Baxter Lie conformal algebra of weight  $\theta$, $(M,\rho) $ be a representation
of $(\A,[\cdot_\lambda \cdot])$ and $\beta:M\rightarrow M$ be a ${\bf
k}[\partial]$-module homomorphism, where $M$ is free as a ${\bf
k}[\partial]$-module. Then $(M^{*c},\rho^*,\beta^*)$
is a representation of $(\A,[\cdot_\lambda \cdot],\B)$ if and only
if $\beta$ satisfies the following condition:
\begin{equation}\label{eqtb}
    \beta(\rho(\B( a))_\lambda v)=\rho(\B (a))_\lambda \beta( v)+\beta(\rho (a)_\lambda \beta (v)) +\theta\rho(a)_\lambda \beta (v),\qquad \forall a\in \A,\;v\in M.
\end{equation}
In particular,  if $\A$ is also free as a ${\bf k}[\partial]$-module, then for a ${\bf k}[\partial]$-module homomorphism $\T:\A\rightarrow \A$, $(\A^{*c},ad^*,\T^*)$ is a representation of $(\A,[\cdot_\lambda \cdot],\B)$ if and only if $\T$ satisfies the following condition.
\begin{equation}\label{eqaa}
    \T([\B( a)_\lambda b])=[\B (a)_\lambda \T( b)]+\T([a_\lambda \T (b)] )+\theta[a_\lambda \T (b)],\qquad \forall a,b\in \A.
\end{equation}

\end{lemma}
\begin{proof}
For all $a\in \A$, $v\in M$ and $f\in M^{\ast c}$, we have
\begin{align*}
    & &\langle  \rho^*(\B (a))_\lambda \beta^* (f)-\beta^*( \rho^*(\B (a))_\lambda f+\rho^*(a)_\lambda \beta^* (f)+\theta\rho^*(a)_\lambda f) ,v \rangle_\mu\\
    &&\quad=-\langle f,\beta(\rho(\B (a))_\lambda v)-\rho(\B (a))_\lambda \beta (v)-(\beta+\theta\id)\rho(a)_\lambda\beta (v)\rangle_{\mu-\lambda}.
\end{align*}
Then the first conclusion follows directly from
Proposition~\ref{prop-dual} and the second conclusion follows from the
first conclusion by taking the adjoint representation.
\end{proof}

\begin{proposition}\label{pp-rbdr}
Let $(\A,[\cdot_\lambda \cdot],\B)$ be a
Rota-Baxter Lie conformal algebra of weight  $\theta$ and $(M,\rho,\T ) $ be a
representation of $(\A,[\cdot_\lambda \cdot],\B)$, where $\A$ and $M$ are
free as   ${\bf k}[\partial]$-modules. Then $(M^{*c},\rho^*,$
$-\theta\id-\T^*)$  
is a representation of $(\A,[\cdot_\lambda \cdot],\B)$. In
particular,  $(\A^{*c},{\rm ad}^*,-\theta \id-\B^*) $ is a
representation of $(\A,[\cdot_\lambda \cdot],\B)$, which is called
the {\bf coadjoint representation} of $(\A,[\cdot_\lambda
\cdot],\B)$.
\end{proposition}
\begin{proof}
Note that Eq. (\ref{eqtb}) holds when $\beta=-\theta\id-\T$. Then this conclusion follows directly from Lemma \ref{m8}.
\end{proof}

\begin{definition}
A {\bf quadratic   Rota-Baxter Lie conformal
algebra of weight  $\theta$} consists of a   Rota-Baxter Lie conformal
algebra $(\A,[\cdot_\lambda\cdot], \B)$  of weight  $\theta$ and a non-degenerate
symmetric invariant conformal bilinear form
$\omega_\lambda(\cdot,\cdot)$ which satisfies the following
condition
\begin{equation}\label{brblca}
    \omega_\lambda(a,\mathfrak{B}(b))+\omega_\lambda(\mathfrak{B}(a),b)+\theta\omega_\lambda(a,b)=0,\qquad \forall a,b\in \A.
\end{equation}
Denote it by $(\A,[\cdot_\lambda\cdot],\mathfrak{B},\omega_\lambda(\cdot,\cdot))$.
\end{definition}

\begin{remark}\label{rmk-1}
Let
$(\A,[\cdot_\lambda\cdot],\mathfrak{B},\omega_\lambda(\cdot,\cdot))$
be a quadratic   Rota-Baxter Lie conformal algebra of weight
$\theta$. By Lemma \ref{bb8}, it is straightforward to show that
$(\A,[\cdot_\lambda\cdot],\widehat{\mathfrak{B}},\omega_\lambda(\cdot,\cdot))$
is also a quadratic   Rota-Baxter Lie conformal algebra of weight
$\theta$. Moreover, the existence of a non-degenerate symmetric
invariant conformal bilinear form $\omega_\lambda(\cdot,\cdot)$
implies that $\A$ is necessarily free as a  ${\bf
k}[\partial]$-module.
\end{remark}

\begin{theorem}\label{flrbl}
Let $(\A,[\cdot_\lambda \cdot],\Delta_r)$ be a factorizable Lie
conformal bialgebra with $I=r_+-r_-$ and $\theta\in {\bf k}$. Then
$(\A,[\cdot_\lambda\cdot],\mathfrak{B},\omega^I_\lambda(\cdot,\cdot))$
is a quadratic    Rota-Baxter Lie conformal
algebra of weight  $\theta$, where the ${\bf k}[\partial]$-module homomorphism
$\mathfrak{B}:\A\rightarrow \A$ and
$\omega^I_\lambda(\cdot,\cdot)$ are defined respectively by
\begin{align}
    \label{b}       &\mathfrak{B}=\theta r_-\circ I^{-1}, \\
    \label{w}       &\omega^I_\lambda(a,b)=\langle I^{-1}(a),b\rangle_\lambda,  \;\;\;\forall  a,b\in \A.
\end{align}

\end{theorem}
\begin{proof}
Let $a,b\in \A$. Then we have
\begin{eqnarray}
    I([I^{-1}(a)_\lambda I^{-1}(b)]^r)&=&(r_+-r_-)([I^{-1}(a)_\lambda I^{-1}(b)]^r)\label{eq-q1}\\
    &=&[(I+r_-)(I^{-1}(a))_\lambda (I+r_-)(I^{-1}(b)) ]-[r_-(I^{-1}(a))_\lambda r_-(I^{-1}(b))]\nonumber\\
    &=&[a_\lambda b]+[r_-(I^{-1}(a))_\lambda b]+[a_\lambda r_-(I^{-1}(b))].\nonumber
\end{eqnarray}
Thus we obtain
\begin{align*}
    [\mathfrak{B}(a)_\lambda \mathfrak{B}(b)]&=[\theta (r_-\circ I^{-1})(a)_\lambda \theta (r_-\circ I^{-1})(b) ]
    =\theta^2 (r_-\circ I^{-1}\circ I)([I^{-1}(a)_\lambda  I^{-1} (b)]^r)\\
    &=\theta^2 (r_-\circ I^{-1})(I([I^{-1}(a)_\lambda  I^{-1} (b)]^r))\\
    &=\theta^2 (r_-\circ I^{-1})([a_\lambda b]+[r_-(I^{-1}(a))_\lambda b]+[a_\lambda r_-(I^{-1}(b))])\\
    &=\mathfrak{B}(\theta[a_\lambda b]+[\mathfrak{B}(a)_\lambda b]+[a_\lambda
    \mathfrak{B}(b)]).
\end{align*}
Hence $\mathfrak{B}$ is a Rota-Baxter operator of weight $\theta$
on $(\A,[\cdot_\lambda \cdot])$. Obviously,
$\omega^I_\lambda(\cdot,\cdot)$ is non-degenerate. By Lemma
\ref{I}, we have
\begin{align*}
    &\omega_\lambda^I(a,b)=\langle I^{-1}(a),b\rangle_\lambda=\langle b,I^{-1}(a)\rangle_{-\lambda}=\langle I^{-1}(b),a\rangle_{-\lambda}
    =\omega_{-\lambda}^I(b,a),\\
    &\omega_{\lambda}^I([a_\mu b],c)=\langle I^{-1}([a_\mu b]),c\rangle_\lambda=-\langle I^{-1}([b_{-\mu-\partial} a]),c\rangle_\lambda =-\langle ad^* (b)_{\lambda-\mu}(I^{-1}(a)),c\rangle_\lambda\\
    &\qquad\qquad\,\,\,\,=\langle I^{-1}(a),[b_{\lambda-\mu}c]\rangle_\mu=\omega_\mu^I(a,[b_{\lambda-\mu
    }c]).
\end{align*}
Hence $\omega^I_\lambda(\cdot,\cdot) $ is a non-degenerate symmetric invariant
conformal bilinear form on $(\A,[\cdot_\lambda\cdot])$. Note that
$r_-^\ast=-r_+$. Therefore we obtain
\begin{eqnarray*}
    &&\omega_{\lambda}^I(a,\mathfrak{B}(b))+\omega_{\lambda}^I(\mathfrak{B}(a),b)+\theta\omega_{\lambda}^I(a,b)\\
    &=&\langle I^{-1} (a),\theta (r_-\circ I^{-1})(b)\rangle_\lambda +\langle I^{-1}(\theta (r_-\circ I^{-1})a,b\rangle_\lambda +\theta\langle I^{-1}a,b\rangle_\lambda\\
    &=&\theta( \langle I^{-1}(a),(r_-\circ I^{-1})(b)+b\rangle_\lambda+\langle (I^{-1}\circ r_-\circ I^{-1})a,b\rangle_\lambda)\\
    &=&\theta\langle a,(I^{-1}\circ r_-\circ I^{-1}+I^{-1}+I^{-1}\circ (-r_+)\circ I^{-1})b\rangle_\lambda=0.
\end{eqnarray*}
So
$(\A,[\cdot_\lambda\cdot],\mathfrak{B},\omega^I_\lambda(\cdot,\cdot))$
is a quadratic    Rota-Baxter Lie conformal
algebra of weight  $\theta$.
\end{proof}

\begin{corollary}
Let $(\A,[\cdot_\lambda \cdot],\Delta_r)$ be a factorizable Lie
conformal bialgebra with $I=r_+-r_-$ and $\theta\in {\bf k}$. Then
$(\A,[\cdot_\lambda\cdot],\Wh{\mathfrak{B}},\omega^I_\lambda(\cdot,\cdot))$
is a quadratic   Rota-Baxter Lie conformal
algebra of weight  $\theta$, where $\mathfrak{B}$ and $\omega^I_\lambda(\cdot,\cdot)$
are given by Eqs. \eqref{b} and \eqref{w} respectively.
\end{corollary}
\begin{proof}
It follows directly from Theorem \ref{flrbl} and Remark \ref{rmk-1}.
\end{proof}
\begin{corollary}\label{cfrbl}
Let $(\A,[\cdot_\lambda \cdot],\Delta_r)$ be a factorizable Lie
conformal bialgebra with $I=r_+-r_-$ and $\theta\in {\bf
k}\backslash\{0\}$. Let $\mathfrak{B}$ be the Rota-Baxter
operator of weight $\theta$ on $(\A, [\cdot_\lambda \cdot])$
defined by Eq. \eqref{b}. Then
$(\A,[\cdot_\lambda\cdot]^\mathfrak{B},\Delta_I)$ is  a Lie
conformal bialgebra, where $$\langle x\otimes y,
\Delta_I(a)\rangle_{\lambda, \mu-\lambda}=\langle [x_\lambda y]^I,
a\rangle_\mu,\;\; \forall a\in \A, x,y\in \A^{*c},$$ and
$$
[x_\lambda y]^I=\theta I^{-1}[(\frac{1}{\theta}I (x))_\lambda  (\frac{1}{\theta}I (y))],\qquad \forall   x,y\in \A^{*c}.
$$
Moreover, $\frac{1}{\theta}I$ is an isomorphism of Lie conformal bialgebras from $(\A^{*c}, \A)$ to $(\A_\mathfrak{B}, (\A^{*c},[\cdot_\lambda\cdot]^I))$.
\end{corollary}
\begin{proof}
Since $I$ is a Lie conformal algebra isomorphism from $(\A^{*c},[\cdot_\lambda\cdot]^r)$ to $(\A,[\cdot_\lambda\cdot])$,  for any $x$, $y\in \A^{*c}$, let $x=I^{-1}(a) $ and $y=I^{-1}(b)$ where $a,b\in \A$. By Eq. (\ref{eq-q1}), we get
\begin{align*}
    &\frac{1}{\theta}I([x_\lambda y]^r)=\frac{1}{\theta}I([I^{-1}(a)_\lambda I^{-1}(b)]^r)\\
    =&\frac{1}{\theta}( [ r_- (I^{-1}(a))_\lambda b]+[a_\lambda r_-(I^{-1}(b))]+[a_\lambda b] )\\
    =&\frac{1}{\theta^2}[a_\lambda b]^\mathfrak{B}=[ \frac{1}{\theta} a_\lambda \frac{1}{\theta} b]^\mathfrak{B}
    =[(\frac{1}{\theta}I(x))_\lambda   (\frac{1}{\theta}I(y))]^\mathfrak{B}.
\end{align*}
Therefore, $\frac{1}{\theta}I$ is a Lie conformal algebra isomorphism from $(\A^{*c},[\cdot_\lambda\cdot]^r)$ to $\A_\mathfrak{B}$.  It is obvious that  $ \frac{1}{\theta}I$ is a Lie conformal algebra isomorphism from $(\A^{*c},[\cdot_\lambda\cdot]^I)$ to $(\A,[\cdot_\lambda\cdot])$. Note that $ (\frac{1}{\theta}I)^\ast=\frac{1}{\theta}I$. Therefore, $\frac{1}{\theta}I$ is a  Lie conformal bialgebra isomorphism from $(\A^{*c},\A)$ to  $(\A_\mathfrak{B},(\A^{*c},[\cdot_\lambda\cdot]^I)) $.
\end{proof}
\begin{example} \label{excdd}
Let $(\A,[\cdot_\lambda \cdot],\Delta) $ be a Lie conformal
bialgebra, where $\A$ is free as a ${\bf k}[\partial]$-module.
Suppose that $(\A^{\ast c}, [\cdot_\lambda \cdot]^\ast)$ is the
        Lie conformal algebra in which $[\cdot_\lambda \cdot]^\ast
        $ is defined by Eq.~(\ref{eq:dual}).
Then by Proposition \ref{LCB-equ}, there is a conformal Manin
triple $(\A\oplus \A^{\ast c}, \A, \A^{\ast c})$.  Recall
\cite[Thoerem 3.10]{L} that there is a quasi-triangular Lie
conformal bialgebra on the Lie conformal algebra $\D
=\A\Join\A^{*c}$ with $r=\sum_ie_i\otimes e^*_i\in \A\otimes
\A^{*c}\subset\D\otimes\D$, where $\{e_i\}^n_{i=1}$ is a ${\bf
k}[\partial]$-basis of $\A$ and $\{e^*_i\}^n_{i=1}$ is the dual
${\bf k}[\partial]$-basis of $\A^{*c}$. Then  $r_+$ and $r_-$ are
given by
\begin{equation*}
    r_+(a,x)=x,\qquad r_-(a,x)=-a,\qquad  \forall a\in\A,\;\;x\in\A^{*c},
\end{equation*}
which means that $I(a,x)=(r_+-r_-)(a,x)=(a,x)$. Thus
$(\D,\D^{*c})$ is a factorizable Lie conformal bialgebra.  Then by
Theorem \ref{flrbl}, there is a quadratic
Rota-Baxter Lie conformal algebra
$(\D,[\cdot_\lambda\cdot],\B,\omega^I_\lambda)$ of weight  $\theta$, where
\begin{align*}
    &\B(a,x)=(\theta  r_-\circ I^{-1})(a,x)=-\theta a,\\
    &\omega^I_{\lambda}((a,x),(b,y))=\langle I^{-1}(a,x),(b,y ) \rangle_\lambda=x_\lambda(b)+y_{-\lambda}(a),\qquad \forall a,b\in \A, \;\;x,y\in \A^{*c}.
\end{align*}
\end{example}

The converse of Theorem \ref{flrbl} also holds, that is, quadratic
Rota-Baxter Lie conformal algebras of nonzero weight    give factorizable Lie conformal bialgebras.

\begin{theorem}\label{rbflca}
Let
$(\A,[\cdot_\lambda\cdot],\mathfrak{B},\omega_\lambda(\cdot,\cdot))$
be a quadratic   Rota-Baxter Lie conformal
algebra of weight  $\theta$. Let
$I_\omega:\A^{*c}\rightarrow
\A$ be the induced ${\bf k}[\partial]$-module isomorphism given by
$$\langle I^{-1}_\omega a,b
\rangle_\lambda=\omega_\lambda(a,b),\;\; \forall a, b\in \A,$$ and
$r\in \A\otimes \A$ with $r_+=\frac{1}{\theta}(\mathfrak{B}+\theta
\id)\circ I_\omega:\A^{*c}\rightarrow \A$. Then  $r= \sum_{j=1}^n \frac{1}{\theta}e_j\otimes (\B+\theta \id )I_\omega(e_j^*)\in \A\otimes \A$  corresponding to $r_+$  is a solution
of the CCYBE in $(\A,[\cdot_\lambda\cdot])$, where $\{e_j\}_{j=1}^n$ is a ${\bf k}[\partial]$-basis of $\A$ and $\{e_j^*\}_{j=1}^n$ is the dual ${\bf k}[\partial]$-basis of $\A^{*c}$.
Thus, there is  a factorizable Lie conformal bialgebra
$(\A,[\cdot_\lambda\cdot],\Delta_r)$. 


\end{theorem}
\begin{proof}
Since $\omega_\lambda(\cdot,\cdot)$ is a symmetric conformal bilinear form,
$(I^{-1}_\omega )^*=I^{-1}_\omega$. By Eq. \eqref{brblca}, we have
$$\mathfrak{B}^*\circ I^{-1}_\omega+I^{-1}_\omega\circ
\mathfrak{B}+\theta I^{-1}_\omega=0. $$ This implies $I_\omega \circ
\mathfrak{B}^*+\mathfrak{B}\circ I_\omega+\theta I_\omega=0$. Thus
$r_+=-\frac{1}{\theta}I_\omega\circ \mathfrak{B}^*$ and
$r_-=-r^*_+=\frac{1}{\theta}\mathfrak{B}\circ I_\omega$. Note that
$r_+-r_-=I_\omega$.

By Remark \ref{rmk-2},  according to
$r_+=\frac{1}{\theta}(\mathfrak{B}+\theta \id)\circ I_\omega $, we
set
\begin{align*}
    r&
    =\sum_j \frac{1}{\theta}e_j\otimes (\B+\theta\id)I_\omega (e_j^*).
\end{align*}
 Let $a\in \A$ and $x,y\in \A^{*c}$. Then we have
\begin{eqnarray*}
    &&\langle [x_\lambda y]^r+[y_{-\lambda-\partial}x]^r,a\rangle_\mu\\
    &&\quad=\langle {\rm ad}^*(r_+(x))_\lambda y-{\rm ad}^*(r_-(y))_{-\lambda-\partial}x+ {\rm ad}^*(r_+(y))_{\mu-\lambda} x-{\rm ad}^*(r_-(x))_{\lambda}y ,a\rangle_\mu\\
    &&\quad=-\langle y,[I_\omega (x)_\lambda a]\rangle_{\mu-\lambda}-\langle x,[I_\omega (y)_{\mu-\lambda}a]\rangle_\lambda\\
    &&\quad=\omega_{\mu-\lambda}(I_\omega (y),[a_{-\mu} I_\omega (x)])-\omega_\lambda(I_\omega (x),[I_\omega
    (y)_{\mu-\lambda}a])=0.
\end{eqnarray*}
Hence $$[x_\lambda y]^r=-[y_{-\lambda-\partial}x]^r,\;\; \forall
x, y\in \A^{*c}.$$
For the Jacobi identity, let $x,y,z\in \A^{*c}$ and $a\in \A$. A similar but longer computation shows
that
$$
\langle  [x_\lambda [y_\mu z]^r]^r-[[x_\lambda y]^r_{\lambda+\mu}z]^r-[y_\mu[x_\lambda z]^r]^r ,a\rangle_\nu=0.
$$
 
Therefore, $
(\A^{*c},[\cdot_\lambda\cdot]^r)$ is a Lie conformal algebra. By
the invariant property of the conformal bilinear form
$\omega_\lambda(\cdot,\cdot)$,  $I_\omega\circ {\rm
ad}^*(a)_\lambda={\rm ad}(a)_\lambda\circ I_\omega$. Then by Lemma
\ref{I} (\ref{item2}), the symmetric part of $r$ is
$\A$-invariant.

Furthermore, we have
\begin{align*}
    I_\omega ([x_\mu y]^r)&=I_\omega ({\rm ad}^*( r_+ (x))_\lambda y -{\rm ad}^*(r_-(y))_{-\lambda-\partial}x)
    =[r_+(x)_\lambda I_\omega(y) ]-[r_-(y)_{-\lambda-\partial} I_\omega(x)]\\
    & = [r_+(x)_\lambda ( r_+-r_-)(y)]-[r_-(y)_{-\lambda-\partial}(r_+-r_-)(x)]\\
    & = [r_+(x)_\lambda r_+(y)]-[r_-(x)_\lambda r_-(y)].
\end{align*}
On the other hand, we have
\begin{align*}
    [I_\omega (x)_\lambda I_\omega(y)]^\B&=[\B( I_\omega (x))_\lambda I_\omega(y) ]+    [I_\omega (x)_\lambda \B(I_\omega(y))]+\theta   [I_\omega (x)_\lambda I_\omega(y)]\\
    & =\theta ( [r_-(x)_\lambda  (r_+-r_-)(y)]+[(r_+-r_-)(x)_\lambda r_-(y)]+[(r_+-r_-)(x)_\lambda (r_+-r_-)(y) ])\\
    & =\theta( [r_+(x)_\lambda r_+(y)]-[r_-(x)_\lambda r_-(y)]).
\end{align*}
Hence $\frac{1}{\theta}I_\omega ([x_\lambda
y]^r)=[\frac{1}{\theta}I_\omega (x)_\lambda
\frac{1}{\theta}I_\omega (y)]^\mathfrak{B}$. Since $\mathfrak{B}$
is a Lie conformal algebra homomorphism from $\A_\mathfrak{B}$ to
$(\A,[\cdot_\lambda\cdot])$, it follows that $r_-$ is a Lie
conformal algebra homomorphism. Thus by Theorem \ref{qtl}, $
(\A,[\cdot_\lambda\cdot],\Delta_r)$ is a factorizable Lie
conformal bialgebra where $\Delta_r$ is induced by
$[\cdot_\lambda\cdot]^r$.
\end{proof}


\begin{definition}
    Let $(\A,\Delta)$ be a Lie conformal coalgebra. A ${\bf k}[\partial]$-module homomorphism $\B: \A \rightarrow\A$ is called a {\bf  Rota-Baxter operator  of weight $\theta$ on $(\A,\Delta)$ } if it satisfies
    \begin{equation}\label{bd}
        (\B\otimes \B)\circ\Delta=(\B\otimes \id+\id\otimes\B )\circ\Delta \circ \B + \theta\Delta \circ \B.
    \end{equation}
    A Lie conformal coalgebra $(\A,\Delta)$ with a Rota-Baxter operator $\B$ of weight $\theta$ is called a {\bf   Rota-Baxter Lie conformal coalgebra of weight  $\theta$}, denoted by $(\A,\Delta,\B)$.
\end{definition}
\begin{remark}\label{rmk-rblcca}
    Let $(\A,\Delta,\B)$ be a   Rota-Baxter Lie conformal coalgebra of weight  $\theta$,
    where $\A$ is free as a ${\bf k}[\partial]$-module. It is straightforward to show that $ \B^*$ is a Rota-Baxter operator of weight $\theta$ on the Lie conformal algebra $(\A^{\ast c},[\cdot_\lambda \cdot]^\ast)$,
     where $[\cdot_\lambda \cdot]^\ast$ is defined by
     Eq.~(\ref{eq:dual}).
\end{remark}
\begin{definition}
    Let $(\A,[\cdot_\lambda\cdot],\Delta)$ be a Lie conformal bialgebra. A ${\bf k}[\partial]$-module homomorphism $\B: \A \rightarrow\A$ is called a {\bf  Rota-Baxter operator  of weight $\theta$ on $(\A,[\cdot_\lambda\cdot],\Delta)$ } if $\B$ is a Rota-Baxter operator of weight $\theta$ on $(\A,[\cdot_\lambda\cdot])$ and $\Wh{\B}$ is a Rota-Baxter operator of weight $\theta$ on $(\A,\Delta)$.
\end{definition}

\begin{proposition}\label{ppbtb}
    Let     $(\A,[\cdot_\lambda\cdot],\B)$ and $(\A',[\cdot_\lambda\cdot]',\B')$ be two   Rota-Baxter Lie conformal
    algebras of weight  $\theta$. Let  $(\A,\A',\rho,\sigma)$ be a matched pair of Lie conformal algebras and  $(\A\Join\A',[\cdot_\lambda\cdot]^\Join)$ be the associated Lie conformal algebra where $[\cdot_\lambda\cdot]^\Join$ defined by Eq. \eqref{mplb}.
    Then $\widetilde{\B}=\B\oplus \B':\A\Join\A'\rightarrow\A\Join\A'$ is a Rota-Baxter operator of weight $\theta$ on the Lie conformal algebra $(\A\Join\A',[\cdot_\lambda\cdot]^\Join)$ if and only if  the following conditions hold:
    \begin{align}
        \label{mrb1}&\rho(\B (a))_\lambda(\B'( x))=\B'(\rho (\B (a ))_\lambda x+\rho(a)_\lambda(\B' (x))+\theta\rho(a)_\lambda x),\\
        \label{mrb2}&\sigma(\B'(x))_\lambda(\B (a))=\B(\sigma (\B' (x))_\lambda a+\sigma(x)_\lambda(\B (a))+\theta\sigma(x)_\lambda a), \qquad \forall a\in\A,x\in\A'.
    \end{align}
\end{proposition}

\begin{proof}
 
    For all $a\in\A$ and $x\in\A'$,  we have
        \begin{align*}
            &[\widetilde{\B}( a,0)_\lambda\widetilde{\B}(0, x) ]^\Join=(-\sigma(\B'x)_{-\lambda-\partial}(\B a),\rho(\B a)_\lambda (\B' x) ),\\
            &\widetilde{\B}(  [\widetilde{\B}(a,0)_\lambda (0,x)   ]^\Join+[(a,0)_\lambda\widetilde{\B}(0,x)]^\Join+\theta[(a,0)_\lambda(0,x)]^\Join      )\\
            =&(-\B(\sigma(\B' x)_{-\lambda-\partial}a+\sigma(x)_{-\lambda-\partial}(\B a)+\theta\sigma(x)_{-\lambda-\partial}a),\B'(\rho(\B a)_\lambda x+\rho(a)_\lambda (\B' x)+\theta\rho(a)_\lambda x)).
        \end{align*}
        Hence $(\A\Join\A',[\cdot_\lambda\cdot]^\Join,\widetilde{\B})$ is a Rota-Baxter Lie conformal algebra of weight  $\theta$ if and only if Eqs. \eqref{mrb1} and \eqref{mrb2} hold.
\end{proof}
\begin{theorem}\label{th-rblca-flcba}
    Let $(\A,[\cdot_\lambda\cdot],\Delta)$ be a Lie conformal bialgebra and $\B$ be a Rota-Baxter operator of weight $\theta$ on  $(\A,[\cdot_\lambda\cdot],\Delta)$  with $\theta\neq 0$, where $\A$ is free as a ${\bf k}[\partial]$-module.
    Suppose that $(\A^{\ast c}, [\cdot_\lambda \cdot]^\ast)$ is the
        Lie conformal algebra in which $[\cdot_\lambda \cdot]^\ast
        $ is defined by Eq.~(\ref{eq:dual}).
    Let $\{e_i\}_{i=1}^n$ be a ${\bf k}[\partial] $-basis of  $\A$ and $\{e_i^*\}_{i=1}^n$ be the dual ${\bf k}[\partial] $-basis of $\A^{*c}$.
    Then $(\A\Join \A^{*c},\Delta_r)$ is a factorizable Lie conformal bialgebra, where $r$ is defined by
    \begin{equation}\label{eq-r}
        r=-\frac{1}{\theta}\sum_i( \B(e_i)\otimes e_i^*+e_i^*\otimes \Wh{\B}(e_i) ).
    \end{equation}
\end{theorem}
\begin{proof}
    By Proposition \ref{LCB-equ}, $(\A\Join \A^{*c},[\cdot_\lambda\cdot]^\Join,\omega_{\lambda}(\cdot,\cdot))$ is a  Lie conformal algebra  with a non-degenerate symmetric invariant conformal bilinear form $\omega_{\lambda}(\cdot,\cdot)$. By Proposition \ref{ppbtb}, $\widetilde{\B}=\B\oplus \Wh{\B^*}$ is a Rota-Baxter operator of weight $\theta$ on $(\A\Join \A^{*c},[\cdot_\lambda\cdot]^\Join)$.  For all $a,b\in \A$ and $x,y\in \A^{*c}$, we have
    \begin{align*}
        &\omega_\lambda((a,x),\widetilde{\B} (b,y))+\omega_\lambda(\widetilde{\B}(a,x),(b,y))+\theta\omega_\lambda((a,x),(b,y))\\
        =&\langle x,\B(b)\rangle_\lambda+\langle \Wh{\B^*}(y),a\rangle_{-\lambda}+\langle \Wh{\B^*},b\rangle_\lambda+\langle y,\B(a)\rangle_{-\lambda}+\theta\langle x,b\rangle_\lambda+\theta \langle
        y,a\rangle_{-\lambda}=0.
    \end{align*}
   Hence $(\A\Join \A^{*c},[\cdot_\lambda\cdot]^\Join,\B,\omega_{\lambda}(\cdot,\cdot))$ is a quadratic Rota-Baxter Lie conformal algebra of weight  $\theta$.
    Then it follows straightforwardly from Theorem \ref{rbflca}.
\end{proof}

\begin{corollary}\label{cor-qrblca-r}
    Let $(\A,[\cdot_\lambda \cdot],\B)$ be a Rota-Baxter Lie conformal algebra of weight  $\theta$ with $\theta\neq 0$, where $\A$ is free as a ${\bf k}[\partial]$-module. Let $\{e_i\}_{i=1}^n$ be a ${\bf k}[\partial] $-basis of  $\A$ and $\{e_i^*\}_{i=1}^n$ be the dual ${\bf k}[\partial] $-basis of $\A^{*c}$. Then $(\A\ltimes_{{\rm ad}^*} \A^{*c},\Delta_r)$ is a factorizable Lie conformal bialgebra where $r$ is defined by Eq. \eqref{eq-r}.
\end{corollary}
\begin{proof}
    Note that in this case the $\lambda$-bracket on $\A^{\ast c}$ is trivial. Then it follows straightforwardly from Theorem \ref{th-rblca-flcba}.
\end{proof}

\subsection{Some induced structures from  factorizable Lie conformal bialgebras}
In this subsection,  since any  Rota-Baxter operator of
nonzero-weight on  a Lie conformal algebra can be rescaled by the
reciprocal of its weight to become a Rota-Baxter operator of
weight 1, we restrict our discussion only to the case of weight 1.
\begin{definition}\label{plca}\cite{YT}
A {\bf post-Lie conformal algebra} $(\A,[\cdot_\lambda\cdot],
\cdot_\lambda\cdot)$   is a Lie conformal algebra
$(\A,[\cdot_\lambda\cdot])$ equipped with a $\lambda$-product
$\cdot_\lambda\cdot: \A\times \A \rightarrow \A[\lambda]$ that
satisfies the conformal sesquilinearity such that the following
conditions hold.
\begin{align}
    \label{eqpl2}& (a_\lambda b)_{\lambda+\mu}c-(b_\mu a)_{\lambda+\mu}c+[a_\lambda b]_{\lambda+\mu}c+b_\mu(a_\lambda c)- a_\lambda(b_\mu c)=0,\\
    \label{eqpl3}& a_\lambda[b_\mu c]=[(a_\lambda b)_{\lambda+\mu}c]+[b_\mu(a_\lambda c)],  \qquad  \forall   a,b,c\in \A.
\end{align}

\end{definition}
\begin{remark}
     Recall \cite{V} that a {\bf post-Lie algebra} $(A,[\cdot,\cdot],\ast)$ is a Lie algebra $(A,[\cdot,\cdot])$ equipped with a binary operation $\ast:A\times A\rightarrow A$ satisfying
    \begin{align*}
        &a\ast[b,c]=[a\ast b,c]+[b,a\ast c],\\
        &(a\ast b-b\ast a+[a,b] )\ast c=a\ast(b\ast c)-b\ast(a\ast c), \qquad \forall a,b,c\in A.
    \end{align*}
\end{remark}

\begin{definition}
A conformal bilinear form
$\omega_{\lambda}(\cdot,\cdot)$ on a post-Lie conformal algebra
$(\A,[\cdot_\lambda\cdot],\cdot_\lambda\cdot)$ is called {\bf
invariant} if Eq. \eqref{cbf1} and the following condition hold.
\begin{equation}\label{cbf2}
    \omega_{\lambda+\mu} (a_\lambda b,c)-\omega_\lambda(a,b_\mu c)=\omega_{\lambda+\mu}(b_\mu a,c)-\omega_\mu(b,a_\lambda c) ,\qquad\forall a,b,c\in \A.
\end{equation}
A {\bf generalized pseudo-Hessian post-Lie conformal algebra}  $(\A,[\cdot_\lambda\cdot],\cdot_\lambda\cdot,\omega_{\lambda}(\cdot,\cdot))$   is a post-Lie  conformal algebra $(\A,[\cdot_\lambda\cdot],\cdot_\lambda\cdot)$ equipped with a non-degenerate symmetric invariant conformal bilinear form $\omega_{\lambda}(\cdot,\cdot)$.
\end{definition}
\begin{remark}
Recall \cite{LBG} that a {\bf generalized pseudo-Hessian post-Lie algebra} $(A,[\cdot,\cdot],\ast, (\cdot,\cdot))$ is a post-Lie algebra $(A,[\cdot,\cdot],\ast)$ equipped with a non-degenerate symmetric bilinear form satisfying
\begin{align*}
    ([a,b],c)=(a,[b,c]),\qquad (a\ast b,c)-(a,b\ast c)=(b\ast a,c) -(b,a\ast c),\qquad \forall a,b,c\in A.
\end{align*}
\end{remark}


\begin{proposition}\label{pp-rb-pl}
Let $(\A,[\cdot_\lambda\cdot],\B)$ be a Rota-Baxter Lie
conformal algebra of weight 1. Define a ${\bf k}$-bilinear map
$\cdot_\lambda\cdot:\A\times \A \rightarrow \A[\lambda]$ by
\begin{equation*}
    a_\lambda b=[\B (a)_\lambda b],\qquad \forall a,b\in \A.
\end{equation*}
Then $(\A,[\cdot_\lambda\cdot],\cdot_\lambda\cdot)$ is a post-Lie conformal algebra, which is called the {\bf induced post-Lie conformal algebra} from   $(\A,[\cdot_\lambda\cdot],\B)$. Furthermore, if there is a non-degenerate symmetric invariant conformal bilinear form $\omega_\lambda(\cdot,\cdot)$ on the Lie conformal algebra $(\A,[\cdot_\lambda\cdot])$, then there is a  generalized pseudo-Hessian post-Lie conformal algebra  $(\A,[\cdot_\lambda\cdot], \cdot_\lambda\cdot, \omega_\lambda(\cdot,\cdot)) $.
\end{proposition}
\begin{proof}
The first conclusion follows directly from \cite{YT}.  
In addition, for all $a,b,c\in \A$, we have
\begin{align*}
    \omega_{\lambda+\mu}(a_\lambda b,c)=-\omega_{\lambda+\mu}([b_\mu \B(a)],c)=-\omega_{\mu}(b,[\B(a)_\lambda c]) =-\omega_\mu(b,a_\lambda
    c).
\end{align*}
Hence Eq. \eqref{cbf2} holds and thus
$(\A,[\cdot_\lambda\cdot],\cdot_\lambda\cdot,
\omega_\lambda(\cdot,\cdot)) $ is a generalized pseudo-Hessian
post-Lie conformal algebra.
\end{proof}

\begin{remark}
A conformal bilinear form $\omega_{\lambda}(\cdot,\cdot)$ on a post-Lie conformal algebra $(\A,[\cdot_\lambda\cdot],\cdot_\lambda\cdot)$ is called {\bf left-invariant} if $\omega_{\lambda}(\cdot,\cdot)$ satisfies Eq. \eqref{cbf1} and
\begin{equation}\label{cbf3}
    \omega_{\lambda+\mu}(a_\lambda b,c)+\omega_\mu(b,a_\lambda c)=0, \qquad \forall a,b,c\in \A.
\end{equation}
 So for the above generalized pseudo-Hessian post-Lie conformal
algebra $(\A,[\cdot_\lambda\cdot],\cdot_\lambda\cdot,
\omega_\lambda(\cdot,\cdot)) $, $\omega_{\lambda}(\cdot,\cdot)$ is
left-invariant on $(\A,[\cdot_\lambda\cdot],\cdot_\lambda\cdot)$.
\end{remark}
Next, we introduce the notion of partial-pre-post-Lie conformal algebras as a conformal generalization of  partial-pre-post-Lie algebras. 
\begin{definition} \label{df-ppplc}
A {\bf partial-pre-post-Lie conformal algebra (pp-post-Lie
conformal algebra)}
$(\A,[\cdot_\lambda\cdot],\lhd_\lambda,\rhd_\lambda)$ is a  Lie
conformal algebra $(\A,[\cdot_\lambda\cdot])$ endowed with two
bilinear $\lambda$-products $\lhd_\lambda,\rhd_\lambda: \A\times\A
\rightarrow \A[\lambda]$ satisfying the conformal sesquilinearity
and the following conditions.
\begin{align}
    \label{pppl1}   &a\lhd_\lambda[b_\mu c]=[a_\lambda b]\lhd_{\lambda+\mu}c-[a_\lambda c]\lhd_{-\mu-\partial}b,\\
    \label{pppl2}   &[a_\lambda (b\lhd_\mu c+c\lhd_{-\mu-\partial} b)]=[a_\lambda c]\lhd_{-\mu-\partial}b+b\lhd_\mu [a_\lambda c],\\
    \label{pppl3}   &a\rhd_\lambda[b_\mu c]-[b_\mu c]\lhd_{-\lambda-\partial}a=[(a\rhd_\lambda b+a\lhd_\lambda b)_{\lambda+\mu}c]+[b_\mu (a\rhd_\lambda c-c\lhd_{-\lambda-\partial}a)],\\
    \label{pppl4}   &\{a_\lambda b\}\rhd_{\lambda+\mu}c=a\rhd_\lambda(b\rhd_\mu c)-b\rhd_\mu(a\rhd_\lambda c)+[b_\mu(a\lhd_\lambda c)]-[a_\lambda (b\lhd_\mu c)]-[a_\lambda b]\lhd_{\lambda+\mu}c,\\
    \label{pppl5}   &(a\rhd_\lambda b-b\lhd_\mu a)\lhd_{\lambda+\mu}c+b\lhd_\mu(a\rhd_\lambda c+a\lhd_\lambda c)=a\rhd_\lambda(b\lhd_\mu c)+[a_\lambda(b\lhd_\mu c)],
\end{align}
where $\{a_\lambda b\}= a\rhd_\lambda b +a\lhd_\lambda
b-b\rhd_{-\lambda-\partial} a-b\lhd_{-\lambda-\partial}
a+[a_\lambda b]$, for all $a,b,c\in \A$.
\end{definition}
\begin{remark}
   Note that the notion of partial-pre-post-Lie (pp-post-Lie) algebras was initially introduced in \cite{LBG}, where they were characterized in terms of representations of post-Lie algebras. This illustrates a ``partial"
   splitting of operations of post-Lie algebras, with the conformal generalization presented in Definition \ref{df-ppplc} extending this structure to the conformal context.

Let $(\A,[\cdot_\lambda\cdot],\lhd_\lambda,\rhd_\lambda)$ be a pp-post-Lie conformal algebra.
\begin{enumerate}
    \item   A pp-post-Lie conformal algebra is called {\bf special},  if $\lhd_\lambda$ is skew-symmetric, i.e., $a\lhd_\lambda b=-b\lhd_{-\lambda-\partial}a $ for all $a,b\in \A$.
    \item If $ \lhd_\lambda$ is trivial, then $(\A,\rhd_\lambda,[\cdot_\lambda\cdot])$ is a post-Lie conformal algebra.
\end{enumerate}

\end{remark}

\begin{proposition}\label{phpl-pppl}
Let $(\A,[\cdot_\lambda\cdot],\cdot_\lambda\cdot,\omega_{\lambda}(\cdot,\cdot))$ be a generalized pseudo-Hessian post-Lie conformal algebra. Then there is a  pp-post-Lie conformal algebra structure $(\A,[\cdot_\lambda\cdot],\lhd_\lambda,\rhd_\lambda)$ on $\A$, where $\lhd_\lambda$ and $\rhd_\lambda$ are defined by
\begin{align}
    \label{pppl-cbf1}   &\omega_{\lambda}(a\lhd_\mu b,c)=\omega_{\mu}(a,c_{-\lambda} b),\\
    \label{pppl-cbf2}   &\omega_\lambda(a\rhd_\mu b,c)=-\omega_{\lambda-\mu}(b,a_\mu c-c_{-\mu-\partial}a),     \qquad \qquad\forall a,b,c\in\A.
\end{align}
Moreover, if $\omega_{\lambda}(\cdot,\cdot) $ is left-invariant, then  $(\A,[\cdot_\lambda\cdot],\lhd_\lambda,\rhd_\lambda)$ is special.
\end{proposition}
\begin{proof}
Let $a,b,c,d\in \A$. Then we have
\begin{align*}
    &\omega_\nu( a\lhd_\lambda[b_\mu c]-[a_\lambda b]\lhd_{\lambda+\mu}c+[a_\lambda c]\lhd_{-\mu-\partial}b,d  )\\
    =&\omega_\lambda(a ,d_{-\nu}[b_\mu c]  )-\omega_{\lambda+\mu}([a_\lambda b],d_{-\nu}c)+\omega_{\nu-\mu}([a_\lambda c],d_{-\nu}b)\\
    =&\omega_\lambda(a,d_{-\nu}[b_\mu c]-[b_\mu(d_{-\nu}c)]+[c_{\nu-\mu-\lambda}(d_{-\nu}b)])\\
    =&\omega_\lambda(a,d_{-\nu}[b_\mu c]-[b_\mu(d_{-\nu}c)]-[(d_{-\nu}b)_{\mu-\nu}c ]
    )=0.
\end{align*}
Hence Eq. \eqref{pppl1} holds. Similarly, we can verify that Eqs.
\eqref{pppl2} -- \eqref{pppl5} hold and thus the first conclusion
holds. If in addition, $\omega_{\lambda}(\cdot,\cdot)$ is
left-invariant, then we have
\begin{align*}
    \omega_{-\lambda}(c\lhd_{-\lambda-\mu} b,a )=\omega_{-\lambda-\mu}(c,a_\lambda b)= \omega_{\lambda+\mu}(a_\lambda b,c)
    = -\omega_\mu(b,a_\lambda
    c)=-\omega_{-\lambda}(b\lhd_{\mu}c,a).
\end{align*}
Hence $\lhd_\lambda$ is skew-symmetric. Therefore,
$(\A,[\cdot_\lambda\cdot],\lhd_\lambda,\rhd_\lambda)$ is special.
\end{proof}

\begin{theorem}\label{th-ias}
Let $(\A,[\cdot_\lambda\cdot],\Delta_r)$ be a factorizable Lie
conformal bialgebra with $I=r_+-r_-$. Then
 there is a generalized pseudo-Hessian post-Lie conformal algebra
 $(\A,[\cdot_\lambda\cdot],\cdot_\lambda\cdot,\omega_{\lambda}(\cdot,\cdot))$
 in which $\cdot_\lambda\cdot$ and $\omega_{\lambda}(\cdot,\cdot)$
 are respectively defined by
    \begin{align*}
        a_\lambda b=[r_-\circ I^{-1}(a)_\lambda b],\qquad \omega_\lambda (a,b)=\langle I^{-1}(a),b\rangle_\lambda, \qquad\forall a,b\in \A.
    \end{align*}
Furthermore, this generalized pseudo-Hessian post-Lie conformal
algebra induces a special pp-post-Lie conformal algebra
    $(\A,[\cdot_\lambda\cdot],\lhd_\lambda,\rhd_\lambda)$ in which
    $\lhd_\lambda$ and $\rhd_\lambda$ are respectively
    defined by
    $$
    a   \lhd_\lambda b=(r_-\circ I^{-1}+ \id)([a_\lambda b]),\qquad a\rhd_\lambda b=[r_-\circ I^{-1}(a)_\lambda b]-a    \lhd_\lambda b,\qquad\forall a,b\in \A.
    $$

\end{theorem}
\begin{proof}
    By Theorem \ref{flrbl}, Proposition \ref{pp-rb-pl} and Proposition \ref{phpl-pppl}, this conclusion holds.
\end{proof}

\section{Factorizable Lie conformal bialgebras via factorizable Gel'fand-Dorfman bialgebras}\label{4}

We give the correspondence between factorizable GDBAs and a  class
of factorizable Lie conformal bialgebras as well as the
correspondence between  quadratic   Rota-Baxter GDAs of nonzero
weight and a class of  quadratic   Rota-Baxter Lie conformal
algebras of the same weight. Moreover, factorizable Novikov
bialgebras induce post-Novikov algebras and special
partial-pre-post-Novikov algebras which correspond to the induced
structures of certain factorizable Lie conformal bialgebras
respectively.

\subsection{A correspondence between factorizable Gel'fand-Dorfman bialgebras and a class of factorizable Lie conformal bialgebras}
Recall that a {\bf Novikov algebra} $(A,\cdot)$ is a vector space $A$ equipped with a binary operation $\cdot$ satisfying
\begin{align}
  \label{na1}  &(a\cdot b)\cdot c-a\cdot(b\cdot c)=(b\cdot a)\cdot c-b\cdot(a\cdot c),\\
    \label{na2}&(a\cdot b)\cdot c=(a\cdot c)\cdot b,\qquad    \forall a,b,c\in A.
\end{align}

A {\bf Gel'fand-Dorfman algebra (GDA)} $(A,\cdot,[\cdot,\cdot])$
is a vector space $A$ with two binary operations $\cdot$,
$[\cdot,\cdot]$ such that $(A, \cdot)$ is a Novikov algebra, $(A,
[\cdot,\cdot])$  is a Lie algebra, and they satisfy the following
compatibility condition.
\begin{equation*}
    [a\cdot b,c]-[a\cdot c,b]+[a,b]\cdot c-[a,c]\cdot b-a\cdot [b,c]=0,\qquad \forall a,b,c\in A.
\end{equation*}

\begin{proposition}  \cite[Theorem 2.2]{X} \label{pp-lc-gd}
    Let $A$ be a vector space with binary operations $\cdot$ and $[\cdot,\cdot]$. Then $(R(A)={\bf k}[\partial] A=\bfk[\partial]\ot_\bfk A,[\cdot_\lambda\cdot])$ is a Lie conformal algebra where the $\lambda$-bracket is defined by
    \begin{equation}\label{eq-lc-gd}
        [a_\lambda b]=\partial(b\cdot a)+\lambda(a\cdot b+b\cdot a)+[a,b],\qquad \forall a,b\in A,
    \end{equation}
    if and only if $(A,\cdot,[\cdot,\cdot])$ is a GDA. $(R(A), [\cdot_\lambda \cdot])$ is called the {\bf Lie conformal algebra corresponding to (the GDA)}  $(A,\cdot,[\cdot,\cdot])$.
    In particular, if $(A,[\cdot,\cdot])$ is abelian, then $(R(A), [\cdot_\lambda \cdot])$  is called the {\bf Lie conformal algebra corresponding to (the Novikov algebra)} $(A,\cdot)$.
\end{proposition}

Recall \cite{WH} the notions of  representations of GDAs.
Let $(A, \cdot,[\cdot,\cdot])$ be a GDA.  Then $(V,l,r,\rho)$ is a
{\bf representation} of $(A, \cdot,[\cdot,\cdot])$ where $V$ is a
vector space and  $l,r,\rho: A \rightarrow {\rm End}_{\bf k} (V) $
are linear maps if and only if $(A\oplus V,\bullet
,[\cdot,\cdot])$ is a GDA where $\bullet$ and $[\cdot,\cdot]$ are
defined by
\begin{align*}
    (a,u)\bullet(b,v)=(a\cdot b,l(a)v+r(b)u),\;\; [(a,u),(b,v)]=([a,b],\rho(a)v-\rho(b)u),\;\;\forall a, b\in A, u, v\in V.
\end{align*}
 The GDA $(A\oplus V,\bullet ,[\cdot,\cdot])$  is called the {\bf semi-direct product} of $(A, \cdot,[\cdot,\cdot])$ and the representation $(V,l,r,\rho)$, which is denoted by $A\ltimes_{(l,r),\rho}V$.  In particular, if the Lie bracket and $\rho$ are trivial, then $(V,l,r)$ is a representation of the Novikov algebra $(A,\cdot)$ \cite{O}.
Note that $(A,L,R,{\rm ad} )$ is a representation of the GDA $(A,\cdot,[\cdot,\cdot])$, where $L,R :A \rightarrow {\rm End}_{\bf k} (A)$ are defined by
$L(a) b=a\cdot b$ and $ R(a)b=b\cdot a$ for all $a,b\in A$. $(A,L,R,{\rm ad} )$ is called the {\bf adjoint representation} of $(A,\cdot,[\cdot,\cdot])$.

Recall \cite{HBG0} that a {\bf Novikov coalgebra}  $(A,\alpha)$ is
a vector space $A$ equipped with a linear map $\alpha$ from $A$ to
$A\otimes A$ satisfying the following conditions.
\begin{align*}
    &(\id\otimes \alpha)\alpha(a)-(\tau\otimes \id)(\id\otimes \alpha)\alpha(a)=(\alpha\otimes \id)\alpha(a)-(\tau\otimes\id)(\alpha\otimes\id)\alpha(a),\\
    &(\tau\otimes \id)(\id\otimes \alpha)\tau\alpha(a)=(\alpha\otimes \id)\alpha(a), \qquad \forall a\in A.
\end{align*}
A {\bf Novikov bialgebra} is a triple $(A,\cdot, \alpha)$, where
$(A,\cdot)$ is a Novikov algebra, $(A,\alpha)$ is a Novikov
coalgebra and they satisfy the following conditions.
\begin{align*}
    &\alpha(a\cdot b)=(R(b)\otimes \id)\alpha(a)+(\id\otimes (L+R)(a))(\id+\tau)\alpha(b),\\
    &(\id-\tau)( (L+R)(a)\otimes \id)\alpha(b)=(\id-\tau)((L+R)(b)\otimes \id)\alpha(a),\\
    &(\id\otimes R(a)-R(a)\otimes \id)(\id+\tau)\alpha(b)=(\id\otimes R(b)-R(b)\otimes \id)(\id+\tau)\alpha(a),\qquad \forall a,b\in A.
\end{align*}
A {\bf Lie coalgebra} $(A,\beta)$ is a vector space $A$ equipped
with  a linear map $\beta$ from $A$ to $A\otimes A$ satisfying the
following conditions.
\begin{align*}
    \beta(a)=-\tau\beta(a),\qquad
    (\beta\otimes \id)\beta(a)=(\id\otimes \beta)\beta(a)-(\tau\otimes \id)(\id\otimes \beta)\beta(a), \qquad \forall a\in A.
\end{align*}
A {\bf Lie bialgebra} is a triple $(A,[\cdot,\cdot],\beta)$ where
$(A,[\cdot,\cdot])$ is a Lie algebra and $(A,\beta)$ is a Lie
coalgebra and they satisfy the following condition.
\begin{equation*}
    \beta([a,b])=({\rm ad}(a)\otimes \id+\id\otimes {\rm ad}(a)) \beta(b)-({\rm ad}(b)\otimes \id+\id\otimes {\rm ad}(b)) \beta(a),\qquad  \forall a,b\in A.
\end{equation*}
\begin{definition}\cite{HBG}
    A {\bf  Gel'fand-Dorfman bialgebra (GDBA)} is a quintuple $(A,\cdot,[\cdot,\cdot],$ $\alpha,\beta)$  where $(A,\cdot, \alpha)$ is a Novikov bialgebra and  $(A,[\cdot,\cdot],\beta)$ is a Lie bialgebra and
    the following compatibility conditions hold.
    \begin{align*}
        &(\id\otimes \beta)\alpha(a)-(\tau\otimes \id)(\id\otimes \alpha)\beta(a)+(\tau\otimes\id)(\id\otimes \beta)\tau\alpha(a)=(\alpha\otimes \id)\beta(a)+(\beta\otimes \id)\alpha(a),\\
        &\beta(b\cdot a )+\alpha([a,b])=( R(a)\otimes \id)\beta(b)+(L(b)\otimes \id+ \id\otimes (L+R)(b))\beta(a)\\
        &\qquad\qquad+({\rm ad}(a)\otimes \id+\id\otimes {\rm ad}(a))\alpha(b)-(\id\otimes {\rm ad}(b))(\id+\tau)\alpha(a),\qquad  \forall a,b\in A.\nonumber
    \end{align*}
    In addition, if  $(A,\cdot,[\cdot,\cdot],\alpha_r,\beta_r)$ is a GDBA, where $r\in A\otimes A$ and $\alpha_r$ and $\beta_r$ are defined by
    \begin{align}
        \label{nr}&\alpha_r(a)=-(L(a)\otimes \id+\id\otimes (L+R)(a))r,\\
        \label{lr}& \beta_r(a)=({\rm ad}(a)\otimes \id+\id\otimes {\rm ad}(a))r, \qquad \forall a\in A,
    \end{align}
    then $(A,\cdot,[\cdot,\cdot],\alpha_r,\beta_r)$ is called {\bf coboundary}.
\end{definition}
\begin{remark}
    Since co-products  $\alpha$ and $\beta$ imply that there is a GDA structure on $A^*$, we always denote the GDBA $(A,\cdot,[\cdot,\cdot],$ $\alpha,\beta)$ by $(A,A^*)$.
\end{remark}
In particular, a  Novikov bialgebra $(A,\cdot, \alpha_r)$  is called {\bf coboundary} when $\alpha_r$ is defined by  Eq. \eqref{nr} for
$r\in A\otimes A$ and a Lie bialgebra $(A,[\cdot,\cdot])$ is called {\bf coboundary} when $\beta_r$ is defined by Eq. \eqref{lr} for $r\in A\otimes A$.

\begin{proposition}\cite[Theorem 2.17]{HBG}\label{lbgb}
    Let $ R(A)={\bf k}[\partial]A$ be the Lie conformal algebra corresponding to a GDA $(A,\cdot,[\cdot,\cdot])$. Suppose that $\alpha$ and $\beta$ are ${\bf k}[\partial]$-module homomorphisms from $ R(A)$ to $ R(A)\otimes  R(A)$ such that $\alpha(A),~\beta(A)\subset A\otimes A$. Then $(  R(A),[\cdot_\lambda\cdot],\Delta)$ is a Lie conformal bialgebra with $\Delta$ defined by
    \begin{equation}\label{eq-lcc-gdc}
        \Delta(a)=(\partial\otimes \id)\alpha(a)-\tau(\partial\otimes \id)\alpha(a)+\beta(a), \qquad \forall a\in A,
    \end{equation}
    if and only if $(A,\cdot,[\cdot,\cdot],\alpha,\beta)$ is a GDBA. We call $(  R(A),[\cdot_\lambda\cdot],\Delta)$ the {\bf Lie conformal bialgebra corresponding to the GDBA} $(A,\cdot,[\cdot,\cdot],\alpha,\beta)$.
\end{proposition}

Recall some notions related to Gel'fand-Dorfman Yang-Baxter
equation. For the details, one refers to \cite{HBG}.
Let $(A,\cdot,[\cdot,\cdot])$ be a GDA and $r=\sum_i a_i\otimes b_i\in A\otimes A$. Define
\begin{align*}
    N(r)&=\sum_{i,j}(a_i\otimes a_j\otimes ({b_i}\cdot b_j) +a_i\otimes (b_i\cdot a_j+a_j\cdot b_i)\otimes b_j+(a_i\cdot a_j)\otimes b_j\otimes b_i),\\
    C(r)&=\sum_{i,j}( [a_i,a_j]\otimes b_i\otimes b_j+a_i\otimes[b_i,a_j]\otimes b_j+a_i\otimes a_j\otimes [b_i, b_j]).
\end{align*}
The equation $N(r)=0$ is called the {\bf Novikov Yang-Baxter
equation (NYBE)} in $(A,\cdot)$. The equation $C(r)=0$ is called
the {\bf classical Yang-Baxter equation (CYBE)} in
$(A,[\cdot,\cdot])$. The equation $C(r)=N(r)=0$ is called the {\bf
Gel'fand-Dorfman Yang-Baxter equation (GDYBE)}  in
$(A,\cdot,[\cdot,\cdot])$.


Recall that a  Lie bialgebra $(A,[\cdot,\cdot],\beta_r)$ is
coboundary if and only if  the symmetric part of $r$ is {\bf
ad-invariant}, that is, $({\rm ad}(a)\otimes \id+\id\otimes {\rm
ad}(a))(r+\tau r)=0$ 
 and ${\rm ad}(a)C(r)=0$, that is,
\begin{equation*}
    ({\rm ad}(a)\otimes \id\otimes \id+\id\otimes {\rm ad}(a)\otimes \id+\id\otimes \id\otimes {\rm ad}(a) )C(r)=0,\qquad \forall a\in A.
\end{equation*}
Recall \cite{CH} that
 if  $r$ is a solution of the NYBE in $(A,\cdot)$ and the
symmetric part of $r$  satisfies the {\bf $(L,R)$-invariant}
condition, i.e.,  $$(L(a)\otimes \id+\id\otimes (L+R)(a))(r+\tau
r)=0,\;\; \forall a\in A,$$ then  $(A, \cdot, \alpha_r)$ is a
Novikov bialgebra where $\alpha_r$ is definde by Eq. (\ref{nr}),
which is called a {\bf quasi-triangular Novikov bialgebra}.
Moreover, if $r$ is skew-symmetric, then $ (A,\cdot,\alpha_r)$ is
called a {\bf triangular Novikov bialgebra}.

\begin{proposition}  \label{gdcob}
    Let $(A,\cdot,[\cdot,\cdot])$ be a GDA. Then $(A,\cdot,[\cdot,\cdot],\alpha_r,\beta_r)$ with $r=\sum_i a_i\otimes b_i\in A\otimes A$ is a coboundary GDBA
    if and only if $(A,\cdot,\alpha_r)$ is a coboundary Novikov algebra, $(A,[\cdot,\cdot],\beta_r)$ is a coboundary Lie algebra and the following equations
    hold.
    \begin{eqnarray*}
        &&(L(a)\otimes \id\otimes \id +\id\otimes(L+R)(a)\otimes\id+\id\otimes\id\otimes(L+R)(a))C(r)-(\id\otimes1\otimes {\rm ad}(a))N(r)\\
        &&\qquad -(\id\otimes {\rm ad}(a)\otimes \id)M(r) - \sum_i( (R(a_i)\otimes \id)({\rm ad}(a)\otimes \id+\id\otimes {\rm ad}(a))(r+\tau r)\otimes b_i   \\
        &&\qquad+ a_i\otimes(\id\otimes {\rm ad}(b_i))((L+R)(a)\otimes\id+\id\otimes L(a)  )(r+\tau r))=0,  \\
        &&(\id\otimes {\rm ad}(b))( (L+R)(a)\otimes\id +\id\otimes L(a))(r+\tau r)=0, \qquad \forall a,b\in A,
    \end{eqnarray*}
where
    $M(r)=N(r)+\tau_{12}N(r)- \sum_i( (L+R)(a_i)\otimes\id+\id\otimes L(a_i)  )(r+\tau r)\otimes b_i$.

\end{proposition}
\begin{proof}
   It  follows directly from \cite[Theorem 3.3]{HBG}.
\end{proof}


\begin{corollary}\label{rmk-m2}
Let $(A,\cdot,[\cdot,\cdot])$ be a GDA and $r\in A\otimes A$. If
$r$ is a solution of the GDYBE  in $(A,\cdot,[\cdot,\cdot])$ and
the symmetric part of $r$  is   $(L,R)$-invariant and {\rm
ad}-invariant, then $(A,\cdot,[\cdot,\cdot],\alpha_r,\beta_r)$ is
a  GDBA, where $\alpha_r$ and $\beta_r$ are defined by Eqs.~{\rm
(\ref{nr})} and {\rm (\ref{lr})} respectively.
\end{corollary}

\begin{proof}
It follows directly from Proposition \ref{gdcob}.
\end{proof}

\begin{definition}
Let $(A,\cdot,[\cdot,\cdot])$ be a GDA. If $r\in A\otimes A$ is a
solution of the GDYBE in $(A,\cdot,[\cdot,\cdot])$ and the
symmetric part of  $r$ is both {\rm ad}-invariant and
$(L,R)$-invariant, then we say that the GDBA
$(A,\cdot,[\cdot,\cdot],\alpha_r,\beta_r)$ with $\alpha_r$ and
$\beta_r$ defined by Eqs. \eqref{nr} and \eqref{lr} respectively
is a {\bf quasi-triangular GDBA}.   In addition,
$(A,\cdot,[\cdot,\cdot],\alpha_r,\beta_r)$ is called a {\bf
triangular GDBA}, if $r$ is a skew-symmetric solution of the GDYBE
in $(A,\cdot,[\cdot,\cdot])$.

\end{definition}



\begin{proposition}
    Let $(A,\cdot,[\cdot,\cdot])$ be a GDA and the symmetric part of $r=\sum_i a_i\otimes b_i \in A\otimes A$ be {\rm ad}-invariant and $(L,R)$-invariant. Then $r$ is a solution of the GDYBE in $(A,\cdot,[\cdot,\cdot])$ if and only if the maps
    $r^G_+$, $r^G_-:A^{*}\rightarrow A$ are GDA homomorphisms, where
    \begin{equation*}
        r^G_+(x)=\sum_i \langle x,a_i\rangle b_i,\qquad  r^G_-(x)=-\sum_i \langle x,b_i\rangle a_i,\qquad \forall x\in A^*.
    \end{equation*}
In this case,     $(A^{*},\cdot_r,[\cdot,\cdot]_r)$ is a GDA, where $\cdot_r$ and $[\cdot,\cdot]_r$ are defined by
    \begin{equation*}
        [x,y]_r={\rm ad}^*(r^G_+(x))y-{\rm ad}^*(r^G_-(y))x,\qquad x\cdot_r y=(L^*+R^*)(r^G_+(x))y-R^*(r^G_-(y))x,\qquad
        \forall x,y\in A^*.
    \end{equation*}
\end{proposition}
\begin{proof}
Let $x$, $y\in A^\ast$ and $a\in A$.   Then it is straightforward
to show that $\langle [x,y]_r, a\rangle =\langle x\otimes y,
\beta_r(a)\rangle$ and $\langle x\cdot_r y, a\rangle =\langle
x\otimes y, \alpha_r(a)\rangle$, where $\alpha_r$ and $\beta_r$
are defined by Eqs. (\ref{nr}) and (\ref{lr}) respectively. By
Remark \ref{rmk-m2}, if $r$ is a solution of the GDYBE in
$(A,\cdot,[\cdot,\cdot])$, then $(A, \cdot, [\cdot,\cdot],
\alpha_r, \beta_r)$ is a GDBA. Therefore,
$(A^{*},\cdot_r,[\cdot,\cdot]_r)$ is a GDA.
    By a straightforward calculation, we have
    \begin{align*}
        &\langle z,r^G_+(x\cdot_r y)-r^G_+(x)\cdot r^G_+(y)\rangle=-\langle x\otimes y\otimes z,N(r)\rangle,\\
        &\langle z,r^G_+([x,y]_r)-[r^G_+(x),r^G_+(y)]\rangle=-\langle x\otimes y\otimes
        z,C(r)\rangle.
    \end{align*}
Hence $r^G_+$ is a homomorphism of GDAs from $(A^*, \cdot_r,
[\cdot,\cdot]_r)$ to $(A, \cdot, [\cdot,\cdot])$ if and only if $r$ is a
solution of the GDYBE in $(A,\cdot,[\cdot,\cdot])$.
By   a straightforward calculation,  $r^G_+$ is a GDA homomorphism if and only if
$r^G_-$ is a GDA homomorphism under the condition that the symmetric part of $r$ is  {\rm ad}-invariant and $(L,R)$-invariant.
\delete{Similarly,
$r^G_-$ is a homomorphism of GDAs from $(A^*, \alpha_r, \beta_r)$
to $(A, \cdot, [\cdot,\cdot])$ if and only if $r$ is a solution of
the GDYBE in $(A,\cdot,[\cdot,\cdot])$.} Thus this conclusion
holds.
\end{proof}

\begin{definition}
    A quasi-triangular GDBA $(A,\cdot,[\cdot,\cdot],\alpha_r,\beta_r)$ is called {\bf factorizable} if the linear map $I_G=r^G_+-r^G_-:(A^*, \cdot_r, [\cdot,\cdot]_r)\rightarrow (A, \cdot, [\cdot,\cdot])$ is an isomorphism of GDAs.
\end{definition}

\begin{theorem}\label{th-flgd}
    Let $(R(A)={\bf k}[\partial]A,[\cdot_\lambda\cdot])$ be the Lie conformal algebra corresponding to a GDA $(A,\cdot,[\cdot,\cdot])$.
    Let $r\in A\otimes A$, which can also be naturally regarded as an element in $ R(A)\otimes  R(A)$.
    Suppose that $\alpha_r$ and $\beta_r$ defined by Eqs. \eqref{nr} and \eqref{lr} respectively are ${\bf k}[\partial]$-module homomorphisms from $ R(A)$ to $R(A)\otimes  R(A)$ such that $ \alpha_r(A),\beta_r(A)\subset A\otimes A$. Then $( R(A),[\cdot_\lambda\cdot],\Delta_r)$ is a quasi-triangular Lie conformal bialgebra if and only if $(A,\cdot,[\cdot,\cdot],\alpha_r,\beta_r)$ is a quasi-triangular GDBA. Moreover, $( R(A),[\cdot_\lambda\cdot],\Delta_r)$ is a factorizable  Lie conformal bialgebra  if and only if $(A,\cdot,[\cdot,\cdot],\alpha_r,\beta_r)$ is a factorizable GDBA.
\end{theorem}
\begin{proof}
    By \cite[Proposition 3.18]{HBG},  $r\in A\otimes A$  is a solution of the CCYBE in $( R(A),[\cdot_\lambda\cdot])$ if and only if $r$ is a solution of the GDYBE in $(A,\cdot,[\cdot,\cdot])$.
 
    Let $r=\sum_i a_i\otimes b_i$. For all $a\in A$ and $x,y\in A^*$, we have
    \begin{align*}
        &\langle x\otimes y,a_{-\partial^{\otimes^2}} (r+\tau r)\rangle _{\lambda,\mu}\\
        =&\sum_i \langle x\otimes y, [a_{-\lambda-\mu}a_i]\otimes b_i+a_i\otimes [a_{-\lambda-\mu} b_i]+[a_{-\lambda-\mu}b_i]\otimes a_i+b_i\otimes [a_{-\lambda-\mu} a_i]\rangle_{\lambda,\mu}\\
        =&-\langle x\otimes y,(\lambda(L(a)\otimes \id+\id\otimes (L+R)(a))+\mu( (L+R)(a)\otimes \id+\id\otimes L(a))) (r+\tau r) \rangle\\
        &+\langle x\otimes y,({\rm ad}(a)\otimes \id+\id\otimes {\rm ad}(a))(r+\tau
        r)\rangle.
    \end{align*}
   Hence the symmetric part of $r$ is $ R(A)$-invariant if and only if the symmetric part of $r$ is {\rm ad}-invariant and $(L,R)$-invariant. Therefore, $( R(A),[\cdot_\lambda\cdot],\Delta_r)$ is a quasi-triangular Lie conformal bialgebra if and only if $(A,\cdot,[\cdot,\cdot],\alpha_r,\beta_r)$ is a quasi-triangular GDBA.

    Note that $r_+\mid_A =r^G_+$ and $r_-\mid_A=r^G_-$. Therefore, $r^G_+$, $r^G_-:A^*\rightarrow A$ are GDA homomorphisms if and only if $r_+,r_-: R(A)^{*c}\rightarrow  R(A)$ are Lie conformal algebra homomorphisms. Thus $I_G:A^*\rightarrow A$ is a GDA isomorphism if and only if $I : R(A)^{*c}\rightarrow  R(A)$ is a Lie conformal algebra isomorphism. Then this conclusion holds.
\end{proof}
\begin{remark}
    By Theorem \ref{th-flgd}, if $(A, [\cdot,\cdot])$ is trivial, we can get a correspondence between a class of quasi-triangular (resp. factorizable) Lie conformal bialgebras  and quasi-triangular (resp. factorizable) Novikov bialgebras.  
\end{remark}


\subsection{Factorizable Gel'fand-Dorfman bialgebras via quadratic  Rota-Baxter  Gel'fand-Dorfman algebras of nonzero weight}

\begin{definition}
A {\bf    Rota-Baxter GDA}
$(A,\cdot,[\cdot,\cdot],\mathfrak{B})$ of weight  $\theta$ consists of  a  GDA
$(A,\cdot,[\cdot,\cdot])$ and a linear map $\B: A\rightarrow A$
such that
\begin{equation}\label{lrb}
    \begin{aligned}
        &\mathfrak{B}(a)\cdot\mathfrak{B}(b)=\mathfrak{B}( \mathfrak{B}(a)\cdot b+a\cdot \mathfrak{B}(b)+\theta a\cdot b ),\\
        &[\B(a),\B(b)]=\B( [\B(a),b]+[a,\B(b)]+\theta[a,b]),\qquad \forall a,b\in A.
    \end{aligned}
\end{equation}
For convenience, sometimes we also denote it by $(A,\mathfrak{B})$.
\end{definition}

Recall \cite{HBG} that a bilinear form $(\cdot,\cdot)$ on a GDA $(A,\cdot,[\cdot,\cdot])$ is called {\bf invariant} if
\begin{align}
&([a,b],c)=(a,[b,c]),\\
\label{pn-bf-1}& (a\cdot b,c)=-(b,a\cdot c+c\cdot a), \qquad \forall a,b,c\in A.
\end{align}
If a GDA $(A,\cdot,[\cdot,\cdot])$ equipped with a non-degenerate symmetric invariant bilinear form $(\cdot,\cdot)$,  then $(A,\cdot,[\cdot,\cdot],(\cdot,\cdot))$ is called  a { \bf quadratic GDA}.
In particular, if the Lie bracket in $(A,\cdot,[\cdot,\cdot],(\cdot,\cdot))$  is trivial, then $(A,\cdot,(\cdot,\cdot))$ is a {\bf quadratic Novikov algebra}. \delete{If Novikov product is  trivial, then $(A,[\cdot,\cdot],(\cdot,\cdot))$ is a {\bf quadratic Lie algebra}.}

\begin{definition}
Let $(A,\cdot,[\cdot,\cdot],\mathfrak{B})$ be a Rota-Baxter GDA of
weight  $\theta$ and $(A,\cdot,[\cdot,\cdot],(\cdot,\cdot))$ be a
quadratic GDA. Then $(A,\cdot,[\cdot,\cdot],\mathfrak{B}
,(\cdot,\cdot))$ is called a {\bf quadratic Rota-Baxter GDA of
weight  $\theta$} if  the following compatibility condition holds.
\begin{equation*}
    (\mathfrak{B}(a),b)+    (a, \mathfrak{B}(b))+\theta(a,b)=0,\qquad \forall a,b\in A.
\end{equation*}
In particular, if the Lie bracket in $(A,\cdot,[\cdot,\cdot],\mathfrak{B} ,(\cdot,\cdot))$ is trivial, then $(A,\cdot,\B,(\cdot,\cdot))$ is a {\bf quadratic   Novikov algebra of weight  $\theta$}.
\end{definition}

\begin{proposition}\label{pp-qrblca-gda}
Let $(R(A), [\cdot_\lambda \cdot])$ be the Lie conformal algebra
corresponding to a  GDA $(A$, $\cdot$, $[\cdot,\cdot])$, and $\B$
be a ${\bf k}[\partial]$-module homomorphism from $R(A)$ to $R(A)$
such that $\B(A)\subset A$. Then $(R(A), [\cdot_\lambda
\cdot],\B)$ is a   Rota-Baxter Lie conformal
algebra  of weight  $\theta$ if and only if $(A, \cdot, [\cdot,\cdot], $ $\B|_A )$ is a
Rota-Baxter GDA of weight  $\theta$. $(R(A), [\cdot_\lambda \cdot],
\B)$  is called {\bf the   Rota-Baxter Lie
conformal algebra of weight  $\theta$ corresponding to $(A, \cdot, [\cdot,\cdot],
\B|_A)$}.

In particular, if there is a conformal bilinear form
$\omega_\lambda(\cdot,\cdot)$ defined by
$$
\omega_\lambda(a,b)=(a,b),\qquad \forall a,b\in A,
$$
then $(R(A), [\cdot_\lambda \cdot],$
$\B,\omega_{\lambda}(\cdot,\cdot))$ is a quadratic
Rota-Baxter Lie conformal algebra of weight  $\theta$ if and only if
$(A, \cdot, [\cdot,\cdot], $ $\B|_A,(\cdot,\cdot) )$ is a
quadratic Rota-Baxter GDA of weight  $\theta$. $(R(A),
[\cdot_\lambda \cdot], \B,\omega_{\lambda}(\cdot,\cdot))$  is
called {\bf the quadratic   Rota-Baxter Lie
conformal algebra  of weight  $\theta$ corresponding to $(A, \cdot, [\cdot,\cdot],
\B|_A,(\cdot,\cdot))$}.
\end{proposition}

\begin{proof}
It is straightforward.
\end{proof}

\begin{theorem}\label{th-rbgd-fgd}
Let $(A,\cdot,[\cdot,\cdot],\alpha_r,\beta_r)$ be a factorizable GDBA with $I_G=r^G_+-r^G_-$ and $\theta\in {\bf k}$. Then $(A,\cdot,[\cdot,\cdot],\B,(\cdot,\cdot)^{I_G})$ is a quadratic Rota-Baxter GDA of weight
 $\theta$, where the linear map $\B:A\rightarrow A$ and the bilinear form $(\cdot,\cdot)^I$ are respectively defined  by
$$
\B=\theta r_-^G\circ I_{G}^{-1},\qquad (a,b)^{I_G}=\langle I_G^{-1}a, b\rangle, \qquad \forall a,b\in A.
$$
Conversely, Let $(A,\cdot,[\cdot,\cdot],\B,(\cdot,\cdot))$ be a quadratic Rota-Baxter GDA of weight  $\theta$ with $\theta\neq 0\in {\bf k}$.
Let $I_G: A^*\rightarrow A$ be the induced vector space isomorphism given by $(a,b)=\langle I_G^{-1} a,b\rangle $ for all $a,b \in \A$.
Then $r=\sum_{j=1}^n \frac{1}{\theta}e_j\otimes (\B+\theta\id)I_G(e_j^*)\in A\otimes A$ is a solution of the GDYBE in $(A,\cdot,[\cdot,\cdot])$, where $\{ e_j\}_{j=1}^n$ is a basis of $A$ and $\{e_j^*\}_{j=1}^n$ is the dual basis of $A^*$. Thus, there is a factorizable GDBA $(A,\cdot,[\cdot,\cdot],\alpha_r,\beta_r)$.
 
\end{theorem}
\begin{proof}

It is straightforward by Theorems \ref{th-flgd}, \ref{flrbl},  \ref{rbflca} and Proposition  \ref{pp-qrblca-gda}.
\end{proof}

 Therefore, we have the following commutative diagram for
$\theta\neq 0$.
\begin{equation*}
\hspace*{-3cm}%
\scalebox{0.8}{
    {\small \xymatrix@C=2cm{
            & \txt{  quadratic    Rota-Baxter GDAs of weight  $\theta$\\  $(A,\cdot,[\cdot,\cdot],\mathfrak{B},(\cdot,\cdot)^{I_G})$} \ar[d]^{\text{Prop.}\ref{pp-qrblca-gda}} \ar@<1ex>[r]^{\quad\qquad \text{Thm.}\ref{th-rbgd-fgd}    }
            &\txt{ factorizable GDBAs\\$(A,\cdot,[\cdot,\cdot],\alpha_r,\beta_r)$} \ar[d]^{\text{Thm.}\ref{th-flgd}}      \ar[l] \\
            & \txt{quadratic     Rota-Baxter Lie conformal algebras\\of weight  $\theta$ \\$(R(A),[\cdot_\lambda\cdot],\B,\omega_\lambda(\cdot,\cdot))$} \ar[u]\ar@<1ex>[r]^{ \text{Thm.} \ref{rbflca}}
            & \txt{factorizable Lie conformal bialgebras\\ $(R(A),[\cdot_\lambda\cdot],\Delta_r)$}\ar[u] \ar[l]^{\text{Thm. }\ref{flrbl} } }               }}
            \end{equation*}

            \begin{definition}
            Let $ ( A,A^*)$ be a GDBA. A linear map $\B:A\rightarrow A$ is called a {\bf Rota-Baxter operator of weight $\theta$ on $ ( A,\cdot,[\cdot,\cdot],\alpha,\beta)$}, if $\B$ is a Rota-Baxter operator of weight $\theta$ on $(A,\cdot,[\cdot,\cdot])$ and $\Wh{\B^*}$ is a Rota-Baxter operator of weight $\theta$ on $(A^*,\cdot^*,[\cdot,\cdot]^*)$.
            \end{definition}

                \begin{theorem}\label{th5}
                Let $ ( A,\cdot_A,[\cdot,\cdot]_A,\alpha,\beta)$ be a  GDBA and $\B$ be a Rota-Baxter operator of weight $\theta$ on $ ( A,\cdot_A,[\cdot,\cdot]_A,\alpha,\beta)$ with   $\theta\neq
                0$. Let $(A^\ast, \cdot_{A^\ast}, [\cdot,\cdot]_{A^{\ast}})$ be the
                GDA structure on $A^*$, where $\cdot_{A^\ast}$ and  $[\cdot,\cdot]_{A^{\ast}}$ are
respectively given by
                 $$\langle x\cdot_{A^{\ast}} y, a\rangle =\langle x\otimes y,\alpha (a)\rangle,\;\;\langle [x,y]_{A^{\ast}} , a\rangle =\langle x\otimes y,\beta (a)\rangle,\;\;\forall x,~y\in A^\ast, a\in
                 A.$$
                   Let $\{e_i\}_{i=1}^n$ be a  basis of $A$ and $\{e_i^*\}_{i=1}^n$ be the dual basis of $A^{*}$. Then
                   $( A\Join A^{*}, \cdot^\Join,[\cdot,\cdot]^\Join,\alpha_r,\beta_r )$ is a factorizable GDBA, where $r$ is defined by Eq. \eqref{eq-r} and $\cdot^\Join,[\cdot,\cdot]^\Join$ are respectively defined by
                \begin{align*}
                    &(a,x)\cdot^\Join (b,y)=(a\cdot_A b+L^*_{A^*}(x)b+R^*_{A^*}(y)a, x\cdot_{A^*} y+L^*_{A}(a)y+R^*_{A}(b)x),\\
                    &[(a,x),(b,y) ]^\Join=([a,b]_A+ {\rm ad}_{A^*}(x)b-{\rm ad}_{A^*}(y)a, [x,y]_{A^*}+{\rm ad}_{A}(a)y-{\rm ad}_{A}(b)x  ),
                \end{align*}
                for all $a,b\in A$ and $ x,y\in A^*$. Moreover, there is a factorizable Lie conformal bialgebra $( R(A\Join A^*),[\cdot_\lambda\cdot],\Delta_r )$, where $[\cdot_\lambda\cdot]$
                and $\Delta_r$ are respectively defined by Eqs. \eqref{eq-lc-gd} and
                \eqref{eq-lcc-gdc}. 
            \end{theorem}
            \begin{proof}

            It follows straightforwardly from \cite[Corollary 4.10]{HBG}, Theorems \ref{th-rbgd-fgd} and \ref{th-flgd}.
            \end{proof}
                \begin{corollary}\label{cor-rbgda-f}
                Let  $(A,\cdot,[\cdot,\cdot],\B)$ be a Rota-Baxter GDA of weight  $\theta$ with $\theta\neq 0$. Let $\{e_i\}_{i=1}^n$ be a  basis of  $A$ and $\{e_i^*\}_{i=1}^n$ be the dual basis
                of $A^{*}$. Then $(A\ltimes_{(L^*+R^*,-R^*),{\rm ad}^*}A^*, \bullet,[\cdot,\cdot],\alpha_r,\beta_r)$ is a factorizable GDBA where $r$ is defined by Eq.
                \eqref{eq-r}. Moreover, there is a factorizable Lie conformal bialgebra $( R(A\ltimes_{(L^*+R^*,-R^*),{\rm ad}^*}A^*),[\cdot_\lambda\cdot],\Delta_r )$, where $[\cdot_\lambda\cdot]$
                and $\Delta_r$ are defined by Eqs. \eqref{eq-lc-gd} and
                \eqref{eq-lcc-gdc}.
            \end{corollary}
            \begin{proof}
                It follows straightforwardly from Theorem \ref{th5}.
 \end{proof}

         We present an example to illustrate a GDBA with a
Rota-Baxter operator of nonzero weight gives a factorizable Lie
conformal bialgebra.

            \begin{example}
                Let $(A={\bf k}e_1\oplus{\bf k}e_2, \cdot,[\cdot,\cdot])$ be the 2-dimensional GDA whose nonzero products are given by
                \begin{align*}
                    e_1\cdot e_1=e_1,\qquad e_2\cdot e_1=e_2,\qquad [e_1,e_2]=-a e_2,
                \end{align*}
                where $a\in {\bf k}\backslash \{0\}$.  Let $\{e_1^*,e_2^*\}$ be the basis of $A^*$ dual to $\{e_1,e_2\}$. Then the nonzero products of  the GDA $(A\ltimes_{(L^*+R^*,-R^*),{\rm ad}^*}A^*, \bullet,[\cdot,\cdot])$ are given by
{\small                \begin{align*}
                    &e_1\bullet e_1=e_1,\qquad e_2\bullet e_1=e_2,\qquad e_1\bullet e_1^*=-2e_1^*,\qquad  e_1\bullet e_2^*=-e_2^*,\\
                    &e_1^*\bullet e_1=e_1^*,\qquad e_2^*\bullet e_1= e_2^*,\qquad e_2\bullet e_2^*=-e_1^*,\qquad [e_1,e_2]=-a e_2, \\
                    &[e_1,e_2^*]=ae_2^*,\qquad [e_2,e_2^*]=-ae_1^*.
                \end{align*}}
                It is straightforward to show that $\B$ is a Rota-Baxter operator of weight
                $1$ on $(A, \cdot,[\cdot,\cdot])$, where $\B(e_1)=-e_1+be_2$ and $\B(e_2)=0$, where $b\in {\bf k}$.
                Then by Corollary \ref{cor-rbgda-f},
                $(A\ltimes_{(L^*+R^*,-R^*),{\rm ad}^*}A^*, \bullet,[\cdot,\cdot],\alpha_r,\beta_r)$ is a factorizable GDBA where  $r=e_1\otimes e_1^*-b e_2\otimes e_1^*+b e_1^*\otimes e_2+e_2^*\otimes e_2$. Thus by Theorem \ref{th-flgd}, $(R(A\ltimes_{(L^*+R^*,-R^*),{\rm ad}^*}A^*),[\cdot_\lambda \cdot],\Delta_r )$ is a factorizable Lie conformal bialgebra where its nonzero $\lambda$-brackets and $\Delta_r$ are given by
                {\small
                \begin{align*}
                    &[{e_1}_\lambda e_1]=(\partial+2\lambda )e_1,\qquad [{e_1}_\lambda e_2]=(\partial+\lambda)e_2-ae_1,\qquad [{e_1}_\lambda e_1^*]=(\partial-\lambda )e_1^*,\\
                    &[{e_1}_\lambda e_2^*]=(\partial+a)e_2^*,\qquad [{e_2}_\lambda e_2^*]=-(\lambda+a)e_1^*,\qquad     \Delta_r (e_1)=b (\partial^{\otimes^2}-a)( e_1^*\otimes e_2 - e_2\otimes e_1^*),\\
                    &\Delta_r(e_2)=(\partial^{\otimes^2}-a)(e_1^*\otimes e_2-e_2\otimes e_1^*),\qquad \Delta_r(e_1^*)=(\id\otimes\partial -\partial\otimes \id)e_1^*\otimes e_1^*,\\
                    &\Delta_r(e_2^*)=b(\partial\otimes \id-\id\otimes \partial )e_1^*\otimes e_1^*.
                \end{align*}}
                \delete{In addition, by Theorem \ref{thfgdrb}, $(\overline{A\ltimes_{(L^*+R^*,-R^*),ad^*}A^* },[\cdot_\lambda \cdot],\omega,\B)$ is a quadratic Rota-Baxter Lie conformal algebra where $ \omega_\lambda (\cdot,\cdot)$ defined by Eq. \eqref{eq-blf}.}
            \end{example}

\subsection{Induced structures from factorizable Novikov bialgebras and their correspondences with those from a class of factorizable Lie conformal bialgebras}
    \begin{definition} \cite{YH}
A {\bf post-Novikov algebra} is a quadruple $(A, \cdot,\lhd,\rhd )$, where $(A,\cdot)$ is a Novikov algebra and  $\lhd,\rhd:A\times A\rightarrow A$ are binary operations  satisfying
{\small
\begin{align}
    \label{eq-pn1}  &(a\rhd c)\lhd b=(a\rhd b+a\lhd b+a\cdot b)\rhd c,\\
    \label{eq-pn2}  &a\rhd(c\lhd b)-(a\rhd c)\lhd b+(c\lhd a)\lhd b=c\lhd(a\lhd b+a\rhd b+a\cdot b),\\
    \label{eq-pn3}&(a\rhd b+a\lhd b+a\cdot b)\rhd c-(b\rhd a +b\lhd a+b\cdot a )\rhd c=a\rhd(b\rhd c)-b\rhd(a\rhd c),\\
        \label{eq-pn4}&(a\lhd b)\lhd c=(a\lhd c)\lhd b,\\
            \label{eq-pn5}&(a\rhd b)\cdot c-a\rhd(b\cdot c)=(b\lhd a)\cdot c-b\cdot(a\rhd c),\\
                \label{eq-pn6}&(b\cdot c)\lhd a-b\cdot(c\lhd a)=(c\cdot b)\lhd a -c\cdot(b\lhd a),\\
    \label{eq-pn7}&(a\rhd b)\cdot c=(a\rhd c)\cdot b,\\
    \label{eq-pn8}&(a\cdot b)\lhd c=(a\lhd c)\cdot b,\qquad \forall a,b,c\in A.
\end{align}}
\end{definition}

\begin{definition}
    A bilinear form $(\cdot,\cdot)$ on a post-Novikov algebra $(A,\cdot,\lhd,\rhd)$ is called {\bf invariant} if $(\cdot,\cdot)$ satisfies Eq.  \eqref{pn-bf-1} and the following condition.
    \begin{equation}\label{pn-bf2}
        (a\rhd b,c)=-(b,a\rhd c+c\lhd a), \qquad \forall a,b,c\in A.
    \end{equation}

    Denote a post-Novikov algebra $(A,\cdot,\lhd,\rhd)$ equipped with a non-degenerate symmetric invariant bilinear form $(\cdot,\cdot)$ by
    $(A,\cdot,\lhd,\rhd,(\cdot,\cdot))$.
\end{definition}
\begin{proposition}\label{pl-n}
Consider a quadruple $ (A,\cdot,\lhd,\rhd)$, where $A$ is a vector space and $\cdot,\lhd,\rhd $ are binary operations on A. Then $(R(A)={\bf k}[\partial ]A, [\cdot_\lambda \cdot], \cdot_\lambda \cdot)$ is a post-Lie conformal algebra where $[\cdot_\lambda \cdot]$ and $\cdot_\lambda \cdot$ are defined by
\begin{align}\label{eq-pn}
    [a_\lambda b]=\partial(b\cdot a)+\lambda(a\cdot b+b\cdot a),\quad a_\lambda b=\partial(b\lhd a)+\lambda(a\rhd b+b\lhd a),\qquad\forall a,b,c\in A.
\end{align}
if and only if  $ (A,\cdot,\lhd,\rhd)$ is a post-Novikov algebra.
$(R(A),[\cdot_\lambda \cdot],\cdot_\lambda \cdot)$ is called the
{\bf post-Lie conformal algebra corresponding to (the post-Novikov
algebra) } $(A,\cdot,\lhd,\rhd)$. In addition, if there is a
conformal bilinear form $\omega_\lambda(\cdot,\cdot)$ defined by
$\omega_\lambda(a,b)= (a,b)$ for all $a,b\in A$, then
$(R(A),[\cdot_\lambda \cdot],\cdot_\lambda \cdot,
\omega_\lambda(\cdot,\cdot))$ is a generalized pseudo-Hessian
post-Lie conformal algebra if and only if  $(A,\cdot,\lhd,\rhd,(\cdot,\cdot))$
is a  post-Novikov algebra equipped with a non-degenerate
symmetric invariant bilinear form $(\cdot,\cdot)$.

\end{proposition}
\begin{proof}
By Proposition \ref{pp-lc-gd}, $(R(A),[\cdot_\lambda \cdot])$ is a
Lie conformal algebra if and only if $(A,\cdot)$ is a Novikov
algebra. Taking Eq. \eqref{eq-pn} into Eq. \eqref{eqpl2}, by
comparing the coefficients of
$\lambda^2,\lambda\partial,\lambda\mu,\partial^2$, we show that
\delete{ For all $a,b,c\in A$, we have
\begin{align*}
    &(a_\lambda b)_{\lambda+\mu}c-(b_\mu a)_{\lambda+\mu}c+[a_\lambda b]_{\lambda+\mu}c+b_\mu(a_\lambda c)- a_\lambda(b_\mu c)\\
    =&( \partial (b\lhd a-a\lhd b+b\cdot a)+\lambda (a\rhd b+b\lhd a+a\cdot b+b\cdot a)-\mu(b\rhd a+a\lhd b)  )_{\lambda+\mu }c\\
    &+b_\mu(\partial (c\lhd a)+\lambda(a\rhd c+c\lhd a))-a_\lambda(\partial (c\lhd b)+\mu(b\rhd c+c\lhd b))\\
    =&\partial^2( (c\lhd b)\lhd a-(c\lhd a)\lhd b  )+ \mu^2(-c\lhd(b\diamond a) - (b\diamond a)\rhd c +b\rhd (c\lhd a)+(c\lhd a)\lhd b)\\
    &+\lambda^2(c\lhd (a\diamond b)+(a\diamond b)\rhd c-a\rhd (c\lhd b)-(c\lhd b)\lhd a   )
    +\lambda \partial (c\lhd (a\diamond b)- a\rhd ( c\lhd b)\\
    &+(a\rhd c+c\lhd a)\lhd b-2(c\lhd b)\lhd a )
    +\mu\partial ( -c\lhd (b\diamond a)+ b\rhd ( c\lhd a)-(b\rhd c+c\lhd b)\lhd a\\
    &+2(c\lhd a)\lhd b   )
    +\lambda\mu(c\lhd (a\diamond b- b\diamond a)+(a\diamond b-b\diamond a)\rhd c -(b\rhd c+c\lhd b)\lhd a\\
    &-a\rhd (b\rhd c+c\lhd b)+b\rhd(a\rhd c+c\lhd a)+(a\rhd c+c\lhd a)\lhd b ),
\end{align*}
where $a\diamond b=a\lhd b+a\rhd b+a\cdot b$.}
 Eq. \eqref{eqpl2} holds if and only if Eqs. \eqref{eq-pn1} -- \eqref{eq-pn4} hold. Similarly, we  show that Eq. \eqref{eqpl3} holds if and only if  Eqs. \eqref{eq-pn5} -- \eqref{eq-pn8} hold. Thus, the first conclusion holds. By  Proposition \ref{pp-qrblca-gda}, we just need to verify that
Eq. \eqref{cbf3} holds if and only if Eq. \eqref{pn-bf2} holds.
For all $a,b, c\in A$, we have
 \begin{align*}
    &\omega_{\lambda+\mu}(a_\lambda b,c)+\omega_\mu(b,a_\lambda c)\\
    =&\omega_{\lambda+\mu}(\partial(b\lhd a)+\lambda(a\rhd b+b\lhd a),c )+\omega_{\mu}(b, \partial(c\lhd a)+\lambda(a\rhd c+c\lhd a)).
\end{align*}
 By comparing the coefficient of $\lambda$, we get the second conclusion.
\end{proof}

\begin{definition}
A   {\bf partial-pre-post-Novikov algebra (pp-post-Novikov algebra)} $(A,\cdot,\nearrow,\searrow,$ $\nwarrow,\swarrow)$ is a Novikov algebra $(A,\cdot)$ endowed with four
 binary operations $\nearrow,\searrow,\nwarrow,\swarrow:A\times A\rightarrow A$ satisfying the following
 conditions.
{\small
\begin{align}
    \label{pppn1}&(c\nwarrow b+ c\nearrow b)\cdot a=(c\cdot a )\nwarrow b+(c\cdot a)\nearrow b,\\
    \label{pppn2}&a\cdot (c\nearrow b+c\nwarrow b)+(c\nearrow b+ c\nwarrow b)\cdot a=(a\nearrow b+a\searrow b)\cdot c+c\cdot (a\nearrow b+a\searrow b),\\
    \label{pppn3}&(a \cdot b) \searrow c + c \nwarrow (a \cdot b) = a \searrow (c \cdot b) + (c \cdot b) \nwarrow a,\\
        \label{pppn4}&(c \cdot b) \nwarrow a = (c \cdot a) \nwarrow b= (c\cdot b)\searrow a,\\
           \label{pppn5}&(a\nwarrow b + a\swarrow b) \cdot c + c \cdot (a\nwarrow b + a\swarrow b)
        = a \swarrow (c \cdot b) + a \searrow (c \cdot b) + (c \cdot b)\nwarrow a+ (c\cdot b)\nearrow a , \\
             \label{pppn6}   &a \searrow (b \cdot c) + (b \cdot c) \nwarrow a = b \searrow (a \cdot c) + (a \cdot c) \nwarrow b,\\
              \label{pppn7}& (a \nwarrow b- a\searrow b ) \cdot c =  c\cdot ( a\searrow b- a\nwarrow b)=0,\\
             \label{pppn8}&(a\swarrow c + a\searrow c) \cdot b = (a\nwarrow b + a\swarrow b) \cdot c,                \\
             \label{pppn9}   & (a\nwarrow b+a\swarrow b )\searrow c+c\nwarrow(a\nwarrow b+a\swarrow b)\\
             &= a\cdot(c\nwarrow b)+(c\nwarrow b)\cdot a+a\swarrow(c\nwarrow b)+(c\nwarrow b)\nearrow a,\nonumber\\
              \label{pppn10}& a\swarrow(c\nearrow b )+(c\nearrow b)\nearrow a+(a\searrow c)\cdot b\\
             &= c\nearrow(a\star b)+c\nwarrow(a\cdot b)+a\cdot (c\nwarrow b)+(c\nwarrow b)\cdot a+(a\swarrow c)\nearrow b,  \nonumber \\
             \label{pppn11}  &(a\nwarrow b+a\nearrow b)\nwarrow c=(a\nwarrow c)\nearrow b+(a\nwarrow c)\cdot b,\\
              \label{pppn12}&(a\star b)\swarrow c+(a\searrow c)\cdot b+(a\cdot b)\searrow c=(a\swarrow c)\nearrow b,
\end{align} }
    {\small
        \begin{flalign}
       \label{pppn13} &a\swarrow (c\searrow b)+(c\searrow b)\nearrow a+a\cdot(c\searrow b)+(c\searrow b)\cdot a\\
       &=c\searrow(a\swarrow b+a\searrow b)+(a\swarrow b+a\searrow b)\nwarrow c,\nonumber\\
                \label{pppn14}&(c\nearrow a+c\searrow a)\searrow b=(c\searrow b)\nearrow a+(c\searrow b)\cdot a,\\
            \label{pppn15}&a\swarrow(b\swarrow c)+(b\swarrow c)\nearrow a+b\cdot(a\swarrow c)+(a\swarrow c)\cdot b\\
        &=b\swarrow(a\searrow c)+(a\swarrow c)\nearrow b+a\cdot(b\swarrow c)+(b\swarrow c)\cdot a,\nonumber\\
        \label{pppn16} & (a\nwarrow b+a\swarrow b)\searrow c=(a\swarrow c+a\searrow c)\nwarrow b,\\
        \label{pppn17}  &(a\nearrow b+a\searrow b)\searrow c+c\nwarrow (a\nearrow b+a\searrow b)\\
        &=a\searrow(c\nearrow b+c\nwarrow b)+(c\nearrow b+c\nwarrow b)\nwarrow a, \nonumber\\
            \label{pppn18}   &(a\nearrow b)\nearrow c+(a\searrow c)\cdot b=(a\nearrow c)\nearrow b+(a\searrow b)\cdot c,\qquad \forall a,b,c\in A,&
\end{flalign} }
where $a\star b= a\nwarrow b+a\nearrow b+a\swarrow b+a\searrow
b+a\cdot b$. In particular, a
pp-post-Novikov algebra is called {\bf special}, if $a\searrow
b=a\nwarrow b$ for all $a$, $b\in A$.
\end{definition}
\begin{proposition}\label{pppl-n}
Let $A$ be a vector space with binary
 $\cdot,\nearrow,\searrow,\nwarrow$ and $\swarrow$.  Then $(R(A)={\bf k}[\partial ]A, [\cdot_\lambda \cdot], \lhd_\lambda, \rhd_\lambda)$ is a pp-post-Lie conformal algebra where $[\cdot_\lambda \cdot]$,  $\lhd_\lambda$ and $ \rhd_\lambda$ are respectively defined by
\begin{equation}\label{pppn}
    \begin{aligned}
        &[a_\lambda b]=\partial(b\cdot a)+\lambda(a\cdot b+b\cdot a),   \qquad   a\lhd_\lambda b=\partial(b\searrow a)+\lambda(a\nwarrow b+b\searrow a),   \\
        & a \rhd_\lambda b=\partial(b\nearrow a)+\lambda(a\swarrow b+b\nearrow a),  \qquad\forall a,b,c\in A.
    \end{aligned}
\end{equation}
if and only if  $(A,\cdot,\nearrow,\searrow,\nwarrow,\swarrow)$  is a pp-post-Novikov algebra. In addition,  $(R(A),[\cdot_\lambda \cdot],\lhd_\lambda, \rhd_\lambda)$ is special if and only if $(A,\cdot,\nearrow,\searrow,\nwarrow,\swarrow)$ is special. $(R(A),[\cdot_\lambda \cdot],\lhd_\lambda, \rhd_\lambda)$ is called the {\bf pp-post-Lie conformal algebra corresponding to (the pp-post-Novikov algebra) } $(A,\cdot,\lhd,\rhd)$.
\end{proposition}
\begin{proof}
Taking Eq. \eqref{pppn} into Eq. \eqref{pppl1}, by comparing
coefficients of $\lambda^2,\partial^2,\lambda\partial,
\lambda\mu$, we show that Eq. \eqref{pppl1} holds if and only if
Eqs. \eqref{pppn3}, \eqref{pppn4} and \eqref{pppn6} hold.
Similarly, we have
\begin{align*}
    & {\rm Eq.} \eqref{pppl2}  \Longleftrightarrow {\rm Eq.} \eqref{pppn7},\;\;
    {\rm Eq.} \eqref{pppl3}\Longleftrightarrow {\rm Eqs.} \eqref{pppn1}, \eqref{pppn2}, \eqref{pppn5} {\rm~ and~} \eqref{pppn8}, \\
    &  {\rm Eq.} \eqref{pppl4}\Longleftrightarrow  {\rm Eqs.} \eqref{pppn10},\eqref{pppn12},\eqref{pppn15} {\rm ~and ~} \eqref{pppn18},\\
    &{\rm Eq.} \eqref{pppl5} \Longleftrightarrow  {\rm Eqs.} \eqref{pppn9},\eqref{pppn11},\eqref{pppn13},\eqref{pppn14},\eqref{pppn16} {\rm ~and~ }\eqref{pppn17}.
\end{align*}
Thus, the first conclusion holds and the second conclusion follows
from
$$
a\lhd_\lambda b=\partial(b\searrow a)+\lambda(a\nwarrow b+b\searrow a)=-b\lhd_{-\lambda-\partial}a=\partial(b\nwarrow a)+\lambda(b\nwarrow a+a\searrow b)
$$
for all $a,b\in A$.
\end{proof}
\begin{theorem}\label{th-fn-pna}
Let $(A ,\cdot,\alpha_r)$ be a factorizable Novikov bialgebra with
$I_G=r^G_+-r^G_-$. Then there is a post-Novikov algebra
$(A,\cdot,\lhd,\rhd,(\cdot,\cdot))$ with a non-degenerate  symmetric invariant
bilinear form $(\cdot,\cdot)$, where $\lhd$, $\rhd $ and
$(\cdot,\cdot)$ are respectively given by
$$ a\lhd b=a\cdot r^G_-\circ {I_G}^{-1}(b),\quad  a\rhd b=r^G_-\circ {I_G}^{-1}(a) \cdot b ,\quad   (a,b)=\langle {I_G}^{-1}a,b\rangle,\qquad   \forall  a,b\in A.$$
Furthermore, there is a special pp-post-Novikov algebra
$(A,\cdot,\nearrow,\searrow,$ $\nwarrow,\swarrow)$ in which
$\nearrow,\searrow,$ $\nwarrow$ and $\swarrow$ are respectively
given by
\begin{align*}
    &a\nwarrow b=a\searrow b=(r^G_-\circ {I_G}^{-1}+\id)(a\cdot b),\quad  a\nearrow b= a\cdot r^G_-\circ {I_G}^{-1}(b)-a\searrow b,\\
    &a\swarrow b=r^G_-\circ {I_G}^{-1}(a)\cdot b-a\searrow b , \qquad \forall a,b\in A.
\end{align*}

\end{theorem}
\begin{proof}
By Theorem \ref{th-rbgd-fgd}, $\B=r_-^G\circ I_G^{-1}$ is a
Rota-Baxter operation of weight 1 on $(A,\cdot)$. Hence, $(A,
\cdot ,\lhd,\rhd)$ is a post-Novikov algebra by \cite[Corollary
2.18]{YH}, where $\lhd $ and $\rhd $ are respectively given by
$$a\lhd b=a\cdot \B(b),\;\; a\rhd b=\B(a)\cdot b,\;\;\forall a,~b\in
A.$$ Let $a,~b,~c\in A$. Since $(\cdot,\cdot)$ is a non-degenerate
symmetric invariant bilinear form on $(A,\cdot)$, we have
\begin{align*}
    (  r^G_-\circ I_G^{-1}(a)\cdot b, c )=-( b, r^G_-\circ I_G^{-1}(a)\cdot c+c\cdot  r^G_-\circ I_G^{-1}(a)    ).
\end{align*}
Hence $(\cdot,\cdot)$ is a non-degenerate symmetric invariant
bilinear form on $(A,\cdot,\lhd,\rhd)$.  By Eq. \eqref{na1}, we have
\begin{align*}
&(c\nwarrow b+ c\nearrow b)\cdot a-(c\cdot a )\nwarrow b-(c\cdot a)\nearrow b
=(c\cdot  r_-^G\circ  I_G^{-1}(b))\cdot a-(c\cdot a)\cdot (r_-^G\circ  I_G^{-1}(b))=0.
\end{align*}
Similarly, we verify that Eqs. \eqref{pppn2} -- \eqref{pppn18}
hold. Thus, $(A,\cdot,\nearrow,\searrow,\nwarrow,\swarrow)$ is a
special pp-post-Novikov algebra.
\end{proof}

Combining Propositions~\ref{phpl-pppl}, \ref{pl-n} and
\ref{pppl-n}, Theorems~\ref{th-flgd} and \ref{th-fn-pna} together,
we have the following commutative diagram.

 \begin{equation*} 
        \hspace*{-1.5cm}%
        \scalebox{0.85}{
            {\small 
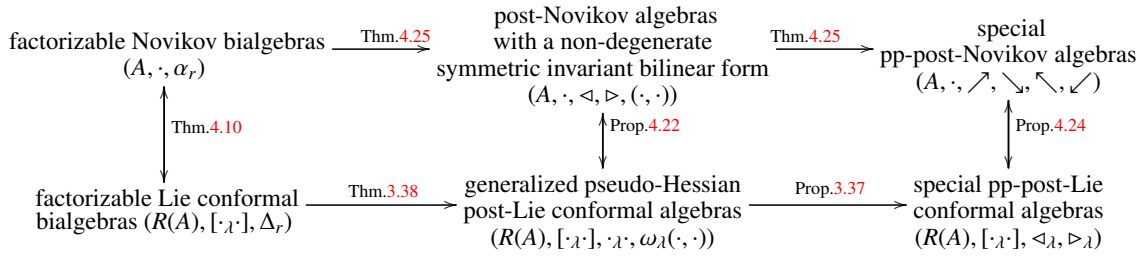
\xymatrix@C=1.5cm{
                    &\txt{ factorizable Novikov bialgebras\\$(A,\cdot,\alpha_r)$  }\ar[r]^{\quad\text{Thm.} \ref{th-fn-pna}}\ar[d]^{\text{Thm.}\ref{th-flgd}}
                    &\txt{    post-Novikov algebras\\  with a non-degenerate \\symmetric invariant
bilinear form    \\$(A,\cdot,\lhd,\rhd,(\cdot,\cdot))$}
\ar[r]^{\text{Thm.}
\ref{th-fn-pna}}\ar[d]^{\text{Prop.}\ref{pl-n}}
                    & \txt{   special\\ pp-post-Novikov algebras\\$(A,\cdot,\nearrow,\searrow,\nwarrow,\swarrow)$  } \ar[d]^{\text{Prop.}\ref{pppl-n}}\\  
                    &\txt{ factorizable Lie conformal\\ bialgebras $(R(A),[\cdot_\lambda\cdot],\Delta_r)$}   \ar[r]^{\text{Thm.}\ref{th-ias} }   \ar[u]
                    &\txt{generalized pseudo-Hessian\\post-Lie conformal algebras\\ $(R(A),[\cdot_\lambda\cdot],\cdot_\lambda\cdot,\omega_\lambda(\cdot,\cdot))$} \ar[r]^{\qquad\text{Prop.}\ref{phpl-pppl}}\ar[u]
                    & \txt{special pp-post-Lie \\conformal algebras\\$(R(A),[\cdot_\lambda\cdot],\lhd_\lambda,\rhd_\lambda)$ } \ar[u]
                     }}}
    \end{equation*}

\begin{example}\label{ex2n}
    Let $(A,\cdot)$ be a 2-dimensional Novikov algebra in \cite{BM} with a basis $\{e_1,e_2 \}$ whose nonzero products are given by
    \begin{align*}
        &e_1\cdot e_2=e_1,\qquad e_2\cdot e_1=-2e_1,\qquad e_2\cdot e_2=e_2.
    \end{align*}
    Define a linear map $\alpha_r: A\rightarrow A\otimes A$ by  $\alpha_r (e_1)=e_1\otimes e_2$ and $ \alpha_r(e_2)=0$.

    It is straightforward  to show that $(A,\cdot,\alpha_r)$ is a Novikov bialgebra with $r=e_1\otimes e_2$ and $I_G=r^G_+-r^G_-$ is an isomorphism of Novikov algebras. Hence $(A,\cdot,\alpha_r)$ is a factorizable Novikov bialgebra.
    Thus there is a   post-Novikov algebra $(A,\cdot,\lhd,\rhd,(\cdot,\cdot))$
    with a non-degenerate symmetric invariant bilinear form $(\cdot,\cdot)$, where the nonzero products of
    $\lhd,\rhd$ and the nonzero values of $(\cdot,\cdot)$  are  respectively given by
    \begin{align*}
       e_2\lhd e_1=2e_1,  \qquad  e_1\rhd e_2=-e_1,\qquad (e_1,e_2)=(e_2,e_1)=1.
    \end{align*}
Moreover, there is a pp-post-Novikov algebra
$(A,\cdot,\nearrow,\searrow,\nwarrow,\swarrow)$, where the nonzero
products of  $\nearrow,\searrow,\nwarrow,\swarrow$ are given by
    \begin{align*}
        e_2\nwarrow e_2=e_2 \searrow e_2=-e_2\nearrow e_2=-e_2\swarrow e_2=e_2,\qquad e_2\nearrow e_1=-2e_1\swarrow e_2=2e_1.
    \end{align*}

    By Theorem \ref{th-flgd}, one gets a factorizable Lie conformal bialgebra $( R(A),[\cdot_\lambda\cdot],\Delta_r)$ where $R(A)={\bf k}[\partial]e_1\oplus {\bf k}[\partial]e_2$ and the nonzero $\lambda$-brackets and co-product are respectively
     defined by
    \begin{align*}
        [{e_1}_\lambda e_2]=-(\lambda+2\partial)e_1,\qquad [{e_2}_\lambda e_2]=(\partial+2\lambda)e_2,\qquad
        \Delta_r(e_1)=\partial e_1\otimes e_2-e_2\otimes \partial e_1. 
    \end{align*}
 There is a generalized pseudo-Hessian post-Lie conformal algebra $(R(A),[\cdot_\lambda\cdot],\cdot_\lambda\cdot,\omega_\lambda(\cdot,\cdot))$, where the  nonzero $\lambda$-products $\cdot_\lambda\cdot$ and
 the nonzero values of $\omega_\lambda(\cdot,\cdot)$
 are  respectively given by
        \begin{align*}
        &{e_1}_\lambda e_2=(\lambda+2\partial)e_1,\qquad \omega_\lambda (e_1,e_2)=\omega_{-\lambda}(e_2,e_1)=1.
    \end{align*}
It corresponds to $(A,\cdot,\lhd,\rhd,(\cdot,\cdot))$ in the sense
of Proposition~\ref{pl-n}. Moreover, there is  a pp-post-Lie
conformal algebra
$(R(A),[\cdot_\lambda\cdot],\lhd_\lambda,\rhd_\lambda)$ whose
nonzero $\lambda$-products $\lhd_\lambda$ and $\rhd_\lambda$ are
given by
    \begin{align*}
        &{e_1}\rhd_\lambda e_2=(\lambda+2\partial)e_1,\qquad e_2\lhd_\lambda e_2=e_2\rhd_\lambda e_2=(\partial+2\lambda)e_2.
    \end{align*}
It corresponds to $(A,\cdot,\nearrow,\searrow,\nwarrow,\swarrow)$
in the sense of Proposition~\ref{pppl-n}.
\end{example}
\noindent {\bf Acknowledgments.} This research is supported by
NSFC (12171129, 12271265, 12261131498, W2412041),  the Zhejiang
Provincial Natural Science Foundation of China (No. LZ25A010004)
and the Fundamental Research Funds for the Central Universities and Nankai Zhide Foundation.

\smallskip

\noindent
{\bf Declaration of interests. } The authors have no conflicts of interest to disclose.

\smallskip

\noindent
{\bf Data availability. } No new data were created or analyzed in this study.

\end{document}